\documentclass{compositio}
\usepackage{amssymb}
\usepackage{amsmath,enumerate}
\usepackage{graphicx}

\theoremstyle{definition}
\newtheorem{Def}{Definition}[section]
\theoremstyle{remark}
\newtheorem{Rem}[Def]{Remark}
\newtheorem{Ex}[Def]{Example}
\newtheorem{Nota}[Def]{Notation}

\newtheorem*{Ack}{Acknowledgments}
\theoremstyle{plain}
\newtheorem{Th}[Def]{Theorem}
\newtheorem{Prop}[Def]{Proposition}
\newtheorem{Lem}[Def]{Lemma}
\newtheorem{Cor}[Def]{Corollary}
\newtheorem{Fact}[Def]{Fact}
\newtheorem{Facts}[Def]{Facts}

\newcommand{\R}{\mathbb{R}}
\newcommand{\Z}{\mathbb{Z}}

\newcommand{\C}{\mathbb{C}}
\renewcommand{\P}{\mathbb{P}}

\newcommand{\CM}{\mathcal{M}}
\newcommand{\CH}{\mathcal{H}}
\newcommand{\FX}{\mathfrak{X}}
\newcommand{\DS}{\displaystyle }

\newcommand{\al}{\alpha }
\newcommand{\be}{\beta }
\newcommand{\ga}{\gamma }
\newcommand{\de}{\delta }
\newcommand{\Ga}{\Gamma }
\newcommand{\De}{\Delta }

\newcommand{\vep}{\varepsilon }
\newcommand{\vph}{\varphi }
\newcommand{\la}{\lambda }
\newcommand{\La}{\Lambda }

\newcommand{\si}{\sigma }

\newcommand{\pa}{\partial }

\makeatletter
\@addtoreset{equation}{section}
\def\theequation{\thesection.\arabic{equation}}
\makeatother
\title[Monodromy of $F_C$]
{The monodromy representation \\
of Lauricella's hypergeometric function $F_C$
}
\author{Yoshiaki Goto}
\email{y-goto@math.kobe-u.ac.jp}
\address[Goto]{
  Department of Mathematics, 
  Graduate School of Science,
  Kobe University, 
  Kobe 657-8501, Japan
}

\date{}

\begin{document}

\begin{abstract}
We study the monodromy representation of 
the system $E_C$ of differential equations satisfied by 
Lauricella's hypergeometric function $F_C$ of 
$m$ variables. 
Our representation space is the
twisted homology group associated with an integral representation of $F_C$. 
We find generators of the fundamental group of 
the complement of the singular locus of $E_C$, and 
give some relations for these generators. 
We express the circuit transformations along these generators, 
by using the intersection forms 
defined on the twisted homology group and its dual. 
\end{abstract}
\maketitle

\section{Introduction}\label{section-intro}
Lauricella's hypergeometric series $F_C$ of $m$ variables $x_1 ,\ldots ,x_m$ 
with complex parameters $a$, $b$, $c_1$, $\ldots$, $c_m$ is defined by 
\begin{align*}
 F_{C} (a,b,c ;x ) 
 =\sum_{n_{1} ,\ldots ,n_{m} =0} ^{\infty } 
 \frac{(a,n_{1} +\cdots +n_{m} )(b,n_{1} +\cdots +n_{m} )}
 {(c_{1} ,n_{1} )\cdots (c_{m} ,n_{m} ) n_{1} ! \cdots n_{m} !} x_{1} ^{n_{1}} \cdots x_{m} ^{n_{m}} ,
\end{align*}
where $x=(x_1 ,\ldots ,x_m),\ c=(c_1 ,\ldots ,c_m)$, 
$c_1 ,\ldots ,c_m \not\in \{ 0,-1,-2,\ldots \}$, and $(c_1 ,n_1)=\Gamma (c_1+n_1)/\Gamma (c_1)$. 
This series converges in the domain 
$$
D_{C} :=\left\{ (x_{1} ,\ldots ,x_{m} ) \in \C ^m  \ \middle| \ \sum _{k=1} ^{m} \sqrt{|x_{k}|} <1  \right\} ,
$$
and admits an Euler-type integral representation (\ref{integral}). 
The system $E_C (a,b,c)$ of differential equations satisfied by  
$F_C (a,b,c;x)$ is a holonomic system 
of rank $2^m$ with the singular locus $S$ given in (\ref{sing-locus}). 
There is a fundamental system of solutions to $E_C (a,b,c)$ in a simply connected domain in $D_C -S$, 
which is given in terms of Lauricella's hypergeometric series $F_C$ with different parameters; 
see (\ref{series-sol}) for their expressions. 

In the case of $m=2$, the series $F_C(a,b,c;x)$ and the system $E_C(a,b,c)$ are called 
Appell's hypergeometric series $F_4(a,b,c;x)$ and system $E_4(a,b,c)$ of differential equations. 
The monodromy representation of $E_4 (a,b,c)$ 
has been studied from several different points of view; see
\cite{Takano}, \cite{Kaneko}, \cite{HU}, and \cite{GM}. 
On the other hand, there were few results of the monodromy representation for general $m$. 
In \cite{Beukers}, Beukers 
studies the monodromy representation of $A$-hypergeometric system and 
gives representation matrices for many kinds of hypergeometric systems 
as examples of his main theorem. 
However, it seems that 
his method is not applicable for Lauricella's $F_C$. 

In this paper, we study the monodromy representation of $E_C (a,b,c)$ 
for general $m$, by using twisted homology groups associated with 
the integral representation (\ref{integral}) of $F_C (a,b,c;x)$ 
and the intersection form defined on the twisted homology groups. 
Our consideration is based on the method for Appell's 
$E_4 (a,b,c)$ in \cite{GM}. 

Let $X$ be the complement of the singular locus $S$. 
The fundamental group of $X$ is generated by $m+1$ loops 
$\rho_0 ,\ \rho_1 ,\ldots ,\ \rho_m$ which satisfy
\begin{align*}
  \rho_i \rho_j =\rho_j \rho_i \quad (1\leq i,j \leq m) ,\quad
  (\rho_0 \rho_k)^2 =(\rho_k \rho_0)^2 \quad (1\leq k \leq m).
\end{align*}
Here, $\rho_k \ (1\leq k \leq m)$ turns the divisor $(x_k=0)$, 
and $\rho_0$ turns the divisor 
$$
\prod _{\vep _{1} ,\ldots ,\vep _{m} =\pm 1} 
\left(1+\sum_{k} \vep _{k} \sqrt{x_{k}} \right) =0  
$$
around the point $\left( \frac{1}{m^2},\ldots, \frac{1}{m^2} \right)$.
In the appendix, we show this claim by applying the Zariski theorem of Lefschetz type. 
Note that for $m=2$, an explicit expression of the fundamental group of $X$ 
is given in \cite{Kaneko}. 

We thus investigate the circuit transformations 
$\CM_i$ along $\rho_i$, for $0\leq i \leq m$. 
We use the $2^m$ twisted cycles $\{ \De_I \}_{I\subset \{ 1,\ldots m \}}$ 
constructed in \cite{G-FC}, 
which represent elements in the $m$-th twisted homology group 
and correspond to the solutions (\ref{series-sol}) to $E_C(a,b,c)$. 
We obtain the representation matrix of $\CM_k \ (1\leq k \leq m)$ 
with respect to the basis $\{ \De_I \}_I$ easily. 
The eigenvalues of $\CM_k$ are $\exp (-2\pi \sqrt{-1} c_k)$ and $1$. 
Both eigenspaces are $2^{m-1}$-dimensional and 
spanned by half subsets of $\{ \De_I \}_I$. 
On the other hand, it is difficult to represent $\CM_0$ directly 
with respect to the basis $\{ \De_I \}_I$. 
Thus we study the structure of the eigenspaces of $\CM_0$. 
We find out that it is quite simple; our main theorem (Theorem \ref{main}) 
is stated as follows. 
The eigenvalues of $\CM_0$ are 
$(-1)^{m-1} \exp (2\pi \sqrt{-1}(c_1 +\cdots +c_m-a-b))$ and $1$. 
The eigenspace $W_0$ of eigenvalue $(-1)^{m-1} \exp (2\pi \sqrt{-1}(c_1 +\cdots +c_m-a-b))$ 
is one-dimensional and spanned by the twisted cycle $D_{1\cdots m}$ defined by 
some bounded chamber. 
Further, the eigenspace $W_1$ of eigenvalue $1$ is characterized as 
the orthogonal complement of $W_0 =\C D_{1\cdots m}$ with respect to the intersection form. 

As a corollary, we express the linear map $\CM_i$ $(0\leq i \leq m)$
by using the intersection form. 
Our expressions are independent of the choice of a basis 
of the twisted homology group. 
To represent $\CM_i$ by a matrix with respect to a given basis, 
it is sufficient to evaluate some intersection numbers. 
In particular, the images of any twisted cycles by $\CM_0$ are 
determined from only the intersection number with the eigenvector $D_{1\cdots m}$; 
see Corollary \ref{CM_0-1}. 
In Section \ref{section-matrix}, we give the simple representation matrix of 
$\CM_i$ with respect to a suitable basis, 
and write down the examples for $m=2$ and $m=3$. 

The irreducibility condition of the system $E_C(a,b,c)$ is 
known to be 
\begin{align*}
  a-\sum _{i\in I} c_i ,\ \ b-\sum _{i \in I} c_i \not\in \Z 
\end{align*}
for any subset $I$ of $\{ 1,\ldots ,m \}$, as is in \cite{HT}. 
Throughout this paper, we assume that the parameters $a,\ b$, and 
$c=(c_1 ,\ldots ,c_m)$ are generic, which means that 
we add other conditions to the irreducibility condition; 
for details, refer to Remark \ref{generic}. 

\begin{Ack}
  The author thanks Professor Keiji Matsumoto 
  for his useful advice and constant encouragement.
  He is also grateful to Professor Jyoichi Kaneko for 
  helpful discussions. 
  % He thanks the referee for suggesting some improvement in 
  % the previous version of the article. 
\end{Ack}

\section{Differential equations and integral representations}
In this section, we collect some facts about Lauricella's $F_C$ 
and the system $E_C$ of 
differential equations that it satisfies. 

\begin{Nota}\label{k}
  \begin{enumerate}[(i)]
  \item Throughout this paper, the letter $k$ always stands for an index running from $1$ to $m$. 
    If no confusion is possible, 
    $\sum_{k=1} ^m$ and $\prod_{k=1} ^m$ are often simply denoted by 
    $\sum$ (or $\sum_k$) and $\prod$ (or $\prod_k$), respectively. 
    For example, under this convention  
    $F_C (a,b,c;x)$ is expressed as 
    \begin{align*}
      F_{C} (a,b,c ;x ) 
      =\sum_{n_{1} ,\ldots ,n_{m} =0} ^{\infty } 
      \frac{(a,\sum n_{k} )(b,\sum n_{k} )}
      {\prod (c_{k} ,n_{k} )\cdot \prod n_{k} ! } \prod x_{k} ^{n_{k}} . 
    \end{align*}
  \item For a subset $I$ of $\{ 1,\ldots ,m \}$, 
    we denote the cardinality of $I$ by $|I|$. 
  \end{enumerate}
\end{Nota}

Let $\pa_k \ (1\leq k \leq m)$ be the partial differential operator with respect to $x_k$. 
We set $\theta_k :=x_k \pa_k$, $\theta :=\sum_k \theta_k$.  
Lauricella's $ F_C (a,b,c;x)$ satisfies differential equations 
$$
\left[ \theta_k (\theta_k+c_k-1)-x_k(\theta +a)(\theta +b)  \right] f(x)=0, 
\quad 1\leq k \leq m.
$$ 
The system generated by them is 
called Lauricella's hypergeometric system $E_C (a,b,c)$ of differential equations. 

\begin{Fact}[\cite{HT}, \cite{L}] \label{solution}
The system $E_{C} (a,b,c)$ is a holonomic system of rank $2^m$ with the singular locus 
\begin{align}
\label{sing-locus}
S:= \left( \prod_k x_{k} \cdot R(x)=0 \right) \subset \C^m ,\quad
R(x_1 ,\ldots ,x_m):=\prod_{\vep_1 ,\ldots ,\vep_m =\pm 1} 
\left( 1+\sum_k \vep_k \sqrt{x_k} \right) . 
\end{align}
If $c_1 ,\ldots ,c_m \not\in \Z$, then 
the vector space of solutions to $E_{C} (a,b,c)$ in a simply connected domain in 
$D_C -S$ is spanned by the following $2^m$ functions: 
\begin{align}
\label{series-sol} 
f_I
% =f_{i_1 \cdots i_r} 
:=\prod _{i \in I} x_i ^{1-c_i} \cdot 
F_C \left( a+|I|-\sum _{i\in I} c_i ,b+|I|-\sum _{i\in I} c_i ,c^I ;
x \right) , 
\end{align}
where $I$ is a subset of $\{ 1,\ldots, m \}$, 
and the row vector $c^I =(c_1^I ,\ldots ,c_m^I)$ of $\C^m$ is defined by 
\begin{align*}
  c^I_k =\left\{
    \begin{array}{cl}
      2-c_k & (k\in I),\\
      c_k & (k\not\in I).
    \end{array}
  \right. 
\end{align*}
\end{Fact}
Note that the solution (\ref{series-sol}) for $I=\emptyset$ is 
$f(=f_{\emptyset})=F_{C} (a,b,c ;x)$, 
and $R(x)$ is an irreducible polynomial of degree $2^{m-1}$ in $x_1 ,\ldots ,x_m$.

\begin{Fact}[Euler-type integral representation, 
Example 3.1 in \cite{AK}] \label{int-Euler}
For sufficiently small positive real numbers $x_1 ,\ldots ,x_m$, 
if $c_1 ,\ldots ,c_m ,a-\sum c_k \not\in \Z$, 
then $F_C (a,b,c;x)$ admits the following integral representation: 
\begin{align}
  \label{integral}
  F_{C} (a,b,c ;x)  
  =&\frac{\Ga (1-a)}{\prod \Ga (1-c_{k})\cdot \Ga (\sum c_{k} -a-m+1)}  \\
  \nonumber
  & \cdot \int _{\De } \prod t_{k} ^{-c_{k} } \cdot (1-\sum t_{k} )^{\sum c_{k} -a-m} 
  \cdot \left( 1-\sum \frac{x_{k}}{t_{k}} \right) ^{-b} dt_{1} \wedge \cdots \wedge dt_{m} , 
\end{align}
where $\De $ is the twisted cycle made by an $m$-simplex, 
in Sections 3.2 and 3.3 of \cite{AK}. 
\end{Fact}
This twisted cycle coincides with $\De_{\emptyset}=\De$ 
introduced in Section \ref{section-solution-cycle}. 
In the case of $m=2$, we show a figure of $\De$ 
in Example \ref{ex-cycle}.

\section{Twisted homology groups and local systems
}\label{section-THG}
For twisted homology groups and the intersection form 
between twisted homology groups, refer to 
\cite{AK}, \cite{Y}, or Section 3 of \cite{G-FC}. 

Put $X:=\C^m-S$ and 
\begin{align*}
  & v(t):=1-\sum_k t_k ,\quad 
  w(t,x):= \prod_k t_k \cdot \left( 1-\sum_k \frac{x_k}{t_k} \right) , \\
  & \FX :=\left\{ (t,x)\in \C^m \times X \ \left| \ 
  \prod_k t_k \cdot v(t) \cdot w(t,x) \neq 0 \right. \right\} .
\end{align*}
There is a natural projection 
$$
pr : \FX \to X ;\ (t,x)\mapsto x,
$$
and we define $T_x := pr^{-1} (x)$ for any $x\in X$. 
We regard $T_x$ as an open submanifold of $\C^m$ 
by the coordinates $t=(t_1 ,\ldots ,t_m)$. 
We consider the twisted homology groups on $T_x$ 
with respect to the multivalued function 
\begin{align*}
  u_x (t) :=& \prod t_k ^{1-c_k +b} \cdot v(t)^{\sum c_{k} -a-m+1} 
  w(t,x) ^{-b} \\
  =& \prod t_k ^{1-c_k } \cdot \left( 1-\sum t_{k} \right)^{\sum c_{k} -a-m+1} 
  \cdot \left( 1-\sum \frac{x_{k}}{t_{k}} \right) ^{-b} 
\end{align*}
(the second equality holds under the coordination of branches). 
We denote the $k$-th twisted homology group by $H_k (T_x, u_x)$, 
and locally finite one by $H^{lf}_k (T_x ,u_x)$. 
\begin{Facts}[\cite{AK}, \cite{G-FC}]
  \label{fact-vanish}
  \begin{enumerate}[(i)]
  \item $H_k (T_x ,u_x)=0 ,\ H^{lf}_k (T_x ,u_x)=0$, 
    for $k \neq m$. 
  \item $\dim H_m (T_x ,u_x)=2^m$. 
  \item The natural map $H_m (T_x ,u_x) \to H^{lf}_m (T_x ,u_x)$ 
    is an isomorphism (the inverse map is called the regularization). 
  \end{enumerate}
\end{Facts}
Hereafter, we identify $H^{lf}_m (T_x ,u_x)$ with $H_m (T_x ,u_x)$, 
and call an $m$-dimensional twisted cycle by a twisted cycle simply. 
Note that the intersection form $I_h$ is defined between 
$H_m (T_x ,u_x)$ and $H_m (T_x ,u_x^{-1})$. 

For $x,x' \in X$ and a path $\tau$ in $X$ from $x$ to $x'$, 
there is the canonical isomorphism 
$$
\tau_* : H_m (T_x ,u_x) \to H_m (T_{x'} ,u_{x'}) .
$$
Hence the family 
$$
\CH:=\bigcup_{x \in X} H_m (T_x ,u_x)
$$
forms a local system on $X$. 

Let $\de$ be a twisted cycle in $T_x$ for a fixed $x$. 
If $x'$ is a sufficiently close point to $x$, there is 
a unique twisted cycle $\de '$ such that 
$\int_{\de '} u_{x'} \vph$ is obtained by 
the analytic continuation of $\int_{\de} u_x \vph$, where
$$
\vph :=\frac{dt_{1} \wedge \cdots \wedge dt_{m} }{\prod t_k \cdot (1-\sum t_{k} )} .
$$ 
Thus we can regard the integration $\int_{\de} u_x \vph$ as a holomorphic function in $x$. 
Fact \ref{int-Euler} means that the integral 
$\int_{\De} u_x \vph 
% \vph :=\frac{dt_{1} \wedge \cdots \wedge dt_{m} }{\prod t_k \cdot (1-\sum t_{k} )} 
$
represents $F_C(a,b,c;x)$ modulo Gamma factors. 
Let $Sol$ be the sheaf on $X$ whose sections are 
holomorphic solutions to $E_C (a,b,c)$. 
The stalk $Sol_x$ at $x\in X$ is the space of 
local holomorphic solutions near $x$. 
\begin{Fact}[\cite{G-FC}]\label{iso-homology-solution}
  For any $x\in X$, 
  $$
  \Phi_x :H_m (T_x ,u_x) \to Sol_x ;\ 
  \de \mapsto \int_{\de} u_x \vph
  $$
  is an isomorphism. 
\end{Fact}

\section{Twisted cycles corresponding to the solutions $f_I$}
\label{section-solution-cycle}
Fact \ref{solution} implies that $Sol_x$ is a $\C$-vector space of 
dimension $2^m$ and spanned by $f_I$'s, for $x\in D_C -S$. 
In \cite{G-FC}, we construct twisted cycles $\De_I$ 
that correspond to $f_I$'s, 
for all subsets $I$ of $\{ 1,\ldots ,m \}$. 
In this section, 
we review the construction of $\De_I$ briefly. 

We construct the twisted cycles $\De_I \in H_m (T_x ,u_x)$, 
for fixed sufficiently small positive real numbers $x_1,\ldots ,x_m$. 
We set $J:=I^c=\{ 1,\ldots ,m \} -I$. 
We consider 
\begin{align*}
  M _I :=\C ^m 
  -\left( \bigcup_k (s_k =0) \cup (v_I =0) \cup (w_I =0) \right) ,
\end{align*}
where $v_I$ and $w_I$ are polynomials in $s_1 ,\ldots ,s_m$ defined by 
\begin{align*}
  v_I :=\prod_{i\in I} s_i \cdot 
  \left( 1-\sum_{i\in I} \frac{x_i}{s_i}-\sum_{j\in J} s_j \right) , \ 
  w_I :=\prod_{j\in J} s_j \cdot 
  \left( 1-\sum_{i\in I} s_i-\sum_{j\in J} \frac{x_j}{s_j} \right) .  
\end{align*}
Let $u_I$ be 
a multivalued function on $M_I$ defined as 
\begin{align*}
u_I :=\prod _{k} s_{k} ^{C_k} \cdot v_I ^A \cdot w_I ^B , 
\end{align*}
where
\begin{align*}
  &A:=\sum c_{k} -a-m+1,\quad B:=-b,  \\
  &C_i :=c_i -1 -A \ (i \in I),\quad C_j :=1-c_j -B \ (j\in J).
\end{align*}
Note that if $I=\emptyset$, then $u_{\emptyset}$ and $M_{\emptyset}$ 
coincide with $u_x$ and $T_x$ in Section \ref{section-THG}, respectively. 
We construct the twisted cycle $\tilde{\De}_I$ 
in $M_I$ with respect to $u_I$. 
Let $\vep$ be a positive real number satisfying $\vep <\frac{1}{m+1}$ and 
$x_k <\frac{\vep ^2}{m}$ (we use the assumption $\vep_1=\cdots =\vep_m =\vep$ 
in Section 4 of \cite{G-FC}). 
We consider the closed subset 
\begin{align*}
  \sigma_{I} :=\left\{ (s_1 ,\ldots ,s_m )\in \R^m \ \Bigg| \ 
    s_k \geq \vep ,\ 
    \begin{array}{l}
      1-\sum_{i \in I} s_i \geq \vep ,\\
      1-\sum_{j \in J} s_j \geq \vep
    \end{array}
  \right\} 
\end{align*}
which is a direct product of an $|I|$-simplex and an $(m-|I|)$-simplex, and  
is contained in the bounded domain 
$$
\left\{ (s_1 ,\ldots ,s_m )\in \R^m \ \Bigg| \ 
  s_k >0 ,\ 
  \begin{array}{l}
    1-\sum_{i \in I} \frac{x_i}{s_i} -\sum_{j\in J} s_j >0 ,\\
    1-\sum_{i \in I} s_i -\sum_{j\in J} \frac{x_j}{s_j} >0
  \end{array}
\right\} .
$$
The orientation of $\sigma _I$ 
is induced from the natural embedding $\R^m \subset \C^m$. 
We construct a twisted cycle from $\sigma _I \otimes u_I$. 
Set $L_{1} :=(s_1 =0)$, $\ldots$, $L_{m} :=(s_m =0)$, 
$L_{m+1} :=(1-\sum_{i\in I} s_i=0)$, $L_{m+2} :=(1-\sum_{j\in J} s_j =0)$, 
and let $U(\subset \R ^m )$ be the bounded chamber surrounded by $L_{1} ,\ldots ,\ L_{m} ,\ L_{m+1} ,\ L_{m+2}$, 
then $\sigma _I$ is contained in $U$. 
Note that we do not consider the hyperplane $L_{m+1}$ (resp. $L_{m+2}$), 
when $I=\emptyset$ (resp. $I=\{ 1,\ldots ,m \}$). 
For $K\subset \{ 1, \ldots ,m+2 \}$, we consider $L_K :=\cap _{p\in K} L_p ,\ U_K :=\overline{U} \cap L_K$ 
and $T_K :=\vep $-neighborhood of $U_K$. 
Then we have 
$$
\sigma _I =U-\bigcup_K T_K .
$$ 
Using these neighborhoods $T_K$, we can construct a twisted cycle $\tilde{\De} _I$ 
in the same manner as Section 3.2.4 of \cite{AK}. 

We briefly explain the expression of $\tilde{\De}_I$. 
For $p=1,\ldots ,m+2$, let $\l_p$ be the $(m-1)$-face of $\si_I$ 
given by $\si_I \cap \overline{T_p}$, 
and let $S_p$ be a positively oriented circle with radius $\vep$ 
in the orthogonal complement of $L_p$ starting from the projection of 
$l_p$ to this space and surrounding $L_p$. 
Then $\tilde{\De}_I$ is written as
\begin{align*}
  \sigma _I \otimes u_I
  +\sum _{ \emptyset \neq K \subset \{ 1,\ldots , m+2\}} 
  \ \prod_{p\in K} \frac{1}{d_p} 
  \cdot \left( \Biggl( \bigcap_{p\in K} l_p \Biggr) \times \prod_{p\in K} S_p \right) \otimes u_I,
\end{align*}
where 
\begin{align*}
  d_i :=\ga_i -1 \ (i\in I),\ d_j :=\ga_j^{-1} -1 \ (j\in J),\ 
  d_{m+1} :=\be^{-1} -1,\ d_{m+2} :=\al^{-1} \prod \ga_k -1, 
\end{align*}
and $\al :=e^{2\pi \sqrt{-1}a} ,\ \be :=e^{2\pi \sqrt{-1}b},\ \ga_k :=e^{2\pi \sqrt{-1}c_k}$. 
We often omit ``$\otimes u_I$''. 
\begin{Ex}\label{ex-cycle}
  In the case of $m=2$ and $I=\emptyset$, we have 
    \begin{align*}
      \tilde{\De} =&\si +\frac{S_1 \times l_1 }{1-\ga_1 ^{-1}} 
      +\frac{S_2 \times l_2}{1-\ga_2 ^{-1}}  
      +\frac{S_4 \times l_4}{1-\al^{-1} \ga_1 \ga_2}  \\
      &+\frac{S_1 \times S_2}{(1\! - \!\ga_1 ^{-1})(1\! -\! \ga_2 ^{-1})} 
      +\frac{S_2 \times S_4}{(1\! -\! \ga_2 ^{-1})(1\! -\! \al^{-1} \ga_1 \ga_2)}   
      +\frac{S_4 \times S_1}{(1\! -\! \al^{-1} \ga_1 \ga_2 )(1\! -\! \ga_1 ^{-1})} , 
    \end{align*}
    where the $1$-chains $l_j$ satisfy $\pa \si =l_1 +l_2 +l_4$ 
    (see Figure \ref{picture-cycle}), 
    and the orientation of each direct product is induced 
    from those of its components. 
    Note that the face $l_3$ does not appear in this case.   
    \begin{figure}[h]
    \centering{
      \includegraphics[scale=1.0]{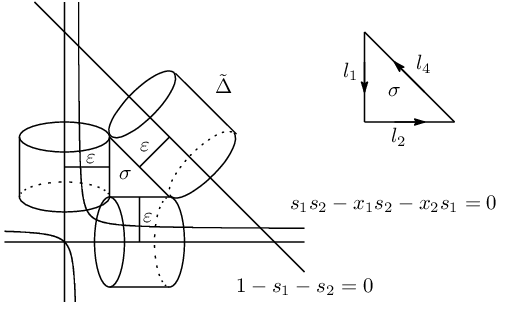} }
    \caption{$\tilde{\De}(=\De )$ for $m=2$. \label{picture-cycle}}
  \end{figure}
\end{Ex}

By using the bijection 
\begin{align*}
  \iota _I :M_I \rightarrow T_x ; \quad &
  \iota _I (s_1 ,\ldots ,s_m ):=(t_1 ,\ldots ,t_m ),\\ 
  &t_i =\frac{x_i}{s_i}\ (i\in I) ,\ t_j =s_j \ (j\in J) ,
\end{align*}
we define the twisted cycle $\De_I$ in $T_x(=M_{\emptyset})$ 
as $\De _I :=(-1)^{|I|} (\iota _I )_{*} (\tilde{\De }_I)$. 
Note that $\iota_I (\si_I)$ is contained in the bounded domain 
$\{ (t_1 ,\ldots ,t_m) \in \R^m \mid 
t_1,\ldots ,t_m,v(t),w(t,x)>0 \}$ which is denoted by $D_{1\cdots m}$ 
in Section \ref{section-MR}. 

We regard $\{ \De_I \} _I$ as the $2^m$ twisted cycles $\De_I$'s 
arranged as 
$(\De ,\De_1 ,\De_2 ,\ldots ,\De_m ,\De_{12} ,\De_{13} ,\ldots ,\De_{1\cdots m})$. 
For a twisted cycle $\de$ with respect to $u_x$, we denote by $\de^{\vee}$ 
the twisted cycle with respect to $u_x^{-1}$, 
which is defined by the same construction as used for $\de$. 
\begin{Fact}[\cite{G-FC}]\label{solution-cycle}
  We have 
  \begin{align*}
    \Phi_x (\De_I)
    =\frac{\prod_{i\in I} \Ga (c_i -1) \cdot \prod_{j\not\in I} \Ga (1-c_j ) 
      \cdot \Ga (\sum_k c_k -a-m+1) \Ga (1-b)}
    {\Ga (\sum_{i\in I} c_i -a-|I|+1) \Ga (\sum_{i\in I} c_i -b-|I|+1)} 
    \cdot f_I. 
  \end{align*}
  The intersection matrix $H:=\left( I_h (\De_I ,\De_{I'} ^{\vee}) \right)_{I,I'}$ 
  is diagonal. Further, the $(I,I)$-entry $H_{I,I}$ of $H$ is
  \begin{align*}
    H_{I,I}
    = (-1) ^{|I|} \cdot \frac{\prod_{j\not\in I} \ga_j 
      \cdot ( \al -\prod_{i\in I} \ga_i) (\be -\prod_{i\in I} \ga_i ) }
    {\prod_k  (\ga_k -1) \cdot ( \al -\prod _k  \ga_k ) (\be -1)} .
  \end{align*}
  Therefore, the $\De_I$'s form a basis of $H_m (T_x ,u_x)$. 
\end{Fact}

\section{Monodromy representation}
\label{section-MR}
Put $\dot{x}:=\left( \frac{1}{2m^2},\ldots ,\frac{1}{2m^2} \right) \in X$. 
For $\rho \in \pi_1 (X,\dot{x})$ and $g \in Sol_{\dot{x}}$, 
let $\rho_* g$ be the analytic continuation of $g$ along $\rho$. 
Since $\rho_* g$ is also a solution to $E_C (a,b,c)$,  
the map $\rho_* :Sol_{\dot{x}} \to Sol_{\dot{x}};\ g\mapsto \rho_* g$ is 
a $\C$-linear automorphism which satisfies 
$(\rho \cdot \rho')_* =\rho'_* \circ \rho_*$ for $\rho,\rho' \in \pi_1 (X,\dot{x})$. 
Here, the composition $\rho \cdot \rho'$ of loops $\rho$ and $\rho'$ is 
defined as the loop going first along $\rho$, 
and then along $\rho'$. 
We thus obtain a representation 
$$
\CM' :\pi_1 (X,\dot{x}) \to GL(Sol_{\dot{x}})
$$
of $\pi_1 (X,\dot{x})$, where $GL(V)$ is 
the general linear group on a $\C$-vector space $V$. 
Since we can identify $Sol_{\dot{x}}$ with $H_m (T_{\dot{x}} ,u_{\dot{x}})$ 
by Fact \ref{iso-homology-solution}, 
the representation $\CM'$ is equivalent to 
$$
\CM :\pi_1 (X,\dot{x}) \to GL(H_m (T_{\dot{x}} ,u_{\dot{x}})) .
$$
Note that for $\rho \in \pi_1 (X,\dot{x})$, the map 
$\CM (\rho) :H_m (T_{\dot{x}} ,u_{\dot{x}}) \to H_m (T_{\dot{x}} ,u_{\dot{x}})$ 
coincides with the canonical isomorphism 
$\rho_* :H_m (T_{\dot{x}} ,u_{\dot{x}}) \to H_m (T_{\dot{x}} ,u_{\dot{x}})$ 
in the local system $\CH$. 
The representation 
$\CM$ (and $\CM'$) is called the monodromy representation, 
which is the main object in this paper. 

For $1\leq k \leq m$, let $\rho_k$ be the loop in $X$ defined by
$$
\rho_k :[0,1] \ni \theta \mapsto 
\left(  \frac{1}{2m^2} ,\ldots ,
  \frac{e^{2\pi \sqrt{-1}\theta}}{2m^2},
  \ldots ,\frac{1}{2m^2} \right) \in X,
$$
where $\frac{e^{2\pi \sqrt{-1}\theta}}{2m^2}$ is the $k$-th 
entry of $\rho_k (\theta)$. 
We take a positive real number $\vep_0$ so that 
$\vep_0 <\min \left\{ \frac{1}{2m^2}, \frac{1}{(m-2)^2}-\frac{1}{m^2} \right\}$, 
and we define the loop $\rho_0$ in $X$ as $\rho_0 :=\tau_0 \rho'_0 \overline{\tau_0}$, 
where 
\begin{align*}
  \tau_0 &:[0,1] \ni \theta \mapsto 
  \left( (1-\theta) \cdot \frac{1}{2m^2}+\theta \cdot \left( \frac{1}{m^2}-\vep_0 \right) \right) 
  (1,\ldots ,1) \in X, \\
  \rho'_0 &:[0,1] \ni \theta \mapsto 
  \left( \frac{1}{m^2} -\vep_0 e^{2\pi \sqrt{-1}\theta}  \right)
  (1,\ldots ,1) \in X,
\end{align*}
and $\overline{\tau_0}$ is the reverse path of $\tau_0$. 
\begin{Rem}
  The loop $\rho_k \ (1\leq k \leq m)$ turns the hyperplane $(x_k=0)$, 
  and $\rho_0$ turns the hypersurface $(R(x)=0)$ around the point 
  $\left( \frac{1}{m^2},\ldots, \frac{1}{m^2} \right)$, positively. 
  Note that $\left( \frac{1}{m^2},\ldots ,\frac{1}{m^2} \right)$ is the nearest to the origin 
  in $(R(x)=0)\cap (x_1=x_2=\cdots =x_m)=\left\{ \frac{1}{m^2}(1,\ldots ,1),
  \frac{1}{(m-2)^2}(1,\ldots ,1),\ldots  \right\}$. 
\end{Rem}

\begin{Th}\label{pi1}
  The loops $\rho_0 ,\rho_1 ,\ldots ,\rho_m$ generate 
  the fundamental group $\pi_1 (X,\dot{x})$. 
  Moreover, if $m\geq 2$, then they satisfy the following relations:
  \begin{align*}
    \rho_i \rho_j =\rho_j \rho_i \quad (1\leq i,j \leq m) ,\quad
    (\rho_0 \rho_k)^2 =(\rho_k \rho_0)^2 \quad (1\leq k \leq m).
  \end{align*}
\end{Th}
\begin{Rem}
  It is shown in \cite{Kaneko} that 
  if $m=2$, then  $\pi_1 (X,\dot{x})$ is the group 
  generated by $\rho_0 ,\rho_1 ,\rho_2$ with 
  the relations in Theorem \ref{pi1}. 
\end{Rem}
We show this theorem in Appendix \ref{section-pi1}.
By this theorem, for the study of the monodromy representation $\CM$, 
it is sufficient to investigate $m+1$ linear maps 
$$
\CM_i :=\CM (\rho_i) \quad (0 \leq i \leq m). 
$$
\begin{Prop}\label{k=1,,m}
  For $1\leq k \leq m$, the eigenvalues of $\CM_k$ are $\ga_k^{-1}$ and $1$. 
  The eigenspace of $\CM_k$ of eigenvalue $\ga_k^{-1}$ is 
  spanned by the twisted cycles 
  $$
  \De_I ,\quad k \in I \subset \{ 1,\ldots ,m \} .
  $$
  That of eigenvalue $1$ is spanned by 
  $$  
  \De_I ,\quad k \not\in I \subset \{ 1,\ldots ,m \} .
  $$
  In particular, both eigenspaces are of dimension $2^{m-1}$. 
\end{Prop}
\begin{proof}
  By Fact \ref{solution-cycle}, the twisted cycle 
  $\De_I$ corresponds to the solution 
  \begin{align*}
    f_I
    =\prod _{i\in I} x_i ^{1-c_i} \cdot 
    F_C \left( a+|I|-\sum _{i\in I} c_i ,b+|I|-\sum _{i\in I} c_i ,c^I ;
      x \right)  
  \end{align*}
  to $E_C (a,b,c)$. 
  Since the series $F_C$ defines a single-valued function around the origin, we have 
  \begin{align*}
    \CM '(\rho_k)(f_I)=\left\{  
      \begin{array}{ll}
        \ga_k^{-1} f_I & (k\in I ) , \\
        f_I & (k\not\in I ).
      \end{array}
      \right.
  \end{align*}
  Therefore, we obtain this proposition. 
\end{proof}
\begin{Cor}\label{CM_k}
  For $1\leq k \leq m$, 
  the linear map $\CM_k :H_m (T_{\dot{x}}, u_{\dot{x}}) \to 
  H_m (T_{\dot{x}}, u_{\dot{x}})$ is expressed as 
  $$
  \CM_k : \de \mapsto 
  \de -(1-\ga_k^{-1})
  \sum_{I\ni k}
  \frac{I_h (\de ,\De_I^{\vee})}
  {I_h (\De_I ,\De_I^{\vee})} \De_I .
  $$
  Further, the representation matrix $M_k$ of $\CM_k$ with respect to 
  the basis $\{ \De_I \}_I$ is the diagonal matrix 
  whose $(I,I)$-entry is 
  $$
  \left\{ 
    \begin{array}{ll}
      \ga_k^{-1} & (I\ni k), \\
      1 & (I\not\ni k). 
    \end{array}
  \right. 
  $$ 
\end{Cor}
\begin{proof}
  We prove the first claim. 
  By Proposition \ref{k=1,,m}, $H_m (T_{\dot{x}}, u_{\dot{x}})$ is 
  decomposed into the direct sum of the eigenspaces: 
  $H_m (T_{\dot{x}}, u_{\dot{x}}) 
  =(\bigoplus_{I\ni k} \C \De_I) \oplus (\bigoplus_{I\not\ni k} \C \De_I)$. 
  Then it is sufficient to show that the claim holds for $\de =\De_I$. 
  This is clear by Fact \ref{solution-cycle} and Proposition \ref{k=1,,m}. 
  The second claim is obvious. 
\end{proof}

For each subset $I\subset \{ 1,\ldots ,m \}$, 
we define a chamber $D_I$ which gives an element in 
$H_m (T_{\dot{x}}, u_{\dot{x}})$. 
For $I=\{ 1,\ldots ,m \}$, we put 
$$
D_{1\cdots m}:=\{ (t_1 ,\ldots ,t_m) \in \R^m \mid 
t_k>0 \ (1\leq k \leq m),\ 
v(t)>0,\ w(t,\dot{x})>0 \} . 
$$
For $I=\emptyset$, we put 
$$
D_{\emptyset}=D:=\{ (t_1 ,\ldots ,t_m) \in \R^m \mid 
t_k <0 \ (1 \leq k \leq m) \} .
$$
For $I \neq \emptyset ,\{ 1,\ldots ,m \}$, we put
$$
D_I:=\left\{ (t_1 ,\ldots ,t_m) \in \R^m \ \left| 
    \begin{array}{l}
      t_i >0 \ (i\in I),\ t_j <0 \ (j \not\in I),\\ 
      v(t)>0,\ (-1)^{m-|I|+1}w(t,\dot{x})>0 
    \end{array}
  \right. \right\} .
$$
The arguments of the factors of $u_{\dot{x}}(t)$ are 
defined as follows. 
$$
\begin{array}{|c|c|c|c|c|}
  \hline 
  & t_i (i\in I) & t_j (j \not\in I) & v(t) & w(t,\dot{x}) 
  \\ \hline 
  D_{1\cdots m} & 0 & - & 0 & 0 
  \\ \hline
  D & - & -\pi & 0 & -m\pi 
  \\ \hline 
  {\rm otherwise} & 0 & -\pi & 0 & -(m-|I|+1)\pi 
  \\ \hline
\end{array}
$$
By the identification of $H^{lf}_m (T_x ,u_x)$ and $H_m (T_x ,u_x)$ 
(see below Fact \ref{fact-vanish}), 
we can consider that the (open) chamber $D_I$ defines an element in $H_m (T_x ,u_x)$. 
Note that if $m=2$, then $D,\ D_1 ,\ D_2$, and $D_{12}$ are equal to 
$\De_6 ,\ \De_7 ,\ \De_8$, and $\De_5$ in \cite{GM}, respectively. 
We state our main theorem. 
\begin{Th}\label{main}
  The eigenvalues of $\CM_0$ are 
  $(-1)^{m-1} \prod_k \ga_k \cdot \al^{-1} \be^{-1}$ and $1$. 
  The eigenspace $W_0$ of $\CM_0$ of eigenvalue 
  $(-1)^{m-1} \prod_k \ga_k \cdot \al^{-1} \be^{-1}$ is 
  spanned by $D_{1\cdots m}$, and hence is one-dimensional. 
  The eigenspace $W_1$ of $\CM_0$ of eigenvalue $1$ is spanned by 
  $$
  D_I ,\quad I \subsetneq \{ 1,\ldots ,m \} ,
  $$
  and expressed as 
  $$
  W_1 =\{ \de \in H_m (T_{\dot{x}}, u_{\dot{x}}) \mid 
  I_h (\de ,D_{1\cdots m}^{\vee})=0 \} .
  $$
  In particular, this space is $(2^m -1)$-dimensional. 
\end{Th}
The proof of this theorem is given in Section \ref{section-proof}. 
\begin{Cor}\label{CM_0-1}
  The linear map $\CM_0 :H_m (T_{\dot{x}}, u_{\dot{x}}) \to 
  H_m (T_{\dot{x}}, u_{\dot{x}})$ is expressed as 
  $$
  \CM_0 : \de \mapsto 
  \de -\left( 1+(-1)^m \prod_k \ga_k \cdot \al^{-1} \be^{-1} \right)
  \frac{I_h (\de ,D_{1\cdots m}^{\vee})}
  {I_h (D_{1\cdots m} ,D_{1\cdots m}^{\vee})} D_{1\cdots m}.
  $$
\end{Cor}
\begin{proof}
  By Theorem \ref{main}, we have 
  $H_m (T_{\dot{x}}, u_{\dot{x}}) =W_0 \oplus W_1 =\C D_{1\cdots m} \oplus W_1$. 
  Then it is sufficient to show that the claim holds for $\de =D_{1\cdots m}$ and 
  $\de \in W_1$. 
  This is clear by Theorem \ref{main}. 
\end{proof}

\begin{Prop}\label{vanishing-cycle}
  We have 
  \begin{align}
    I_h (D_{1\cdots m},\De_I^{\vee})
    =I_h (\De_I,\De_I^{\vee}) 
    =I_h (\De_I,D_{1\cdots m}^{\vee}).
    \label{D-Delta} 
  \end{align}
  Thus we obtain 
  \begin{align}
    & D_{1\cdots m} =\sum_{I \subset \{ 1,\ldots ,m \}} \De_I ,
    \label{D-sum} \\
    &I_h (D_{1\cdots m} ,D_{1\cdots m}^{\vee} )
    =\frac{\al \be +(-1)^m \prod_k \ga_k}{(\be-1)(\al-\prod_k \ga_k )}.
    \label{D-self-int}
  \end{align}
\end{Prop}
This proposition is also proved in Section \ref{section-proof}. 
By this proposition, we obtain the following corollary. 
\begin{Cor}\label{CM_0-2}
  The linear map $\CM_0$ is expressed as 
  $$
  \CM_0 : \de \mapsto 
  \de -\frac{(\be-1)(\al -\prod_k \ga_k )}{\al \be}
  I_h (\de ,D_{1\cdots m}^{\vee}) D_{1\cdots m}. 
  $$
  Let $M_0$ be
  the representation matrix of $\CM_0$ with respect to 
  the basis $\{ \De_I \} _I$. Then we have  
  $$
  M_0 =E_{2^m} -\frac{(\be-1)(\al -\prod_k \ga_k )}{\al \be} NH ,
  $$
  where $E_{2^m}$ is the unit matrix of size $2^m$, 
  $N$ is the $2^m \times 2^m$ matrix with all entries $1$, and 
  $H=\left( I_h (\De_I ,\De_{I'}^{\vee}) \right)_{I,I'}$ is the intersection matrix 
  given in Fact \ref{solution-cycle}.
\end{Cor}
\begin{proof}
  The expression of $\CM_0$ follows immediately from Corollary \ref{CM_0-1} and 
  (\ref{D-self-int}). 
  To obtain the representation matrix, we have to show that 
  the representation matrix of the linear map 
  $\de \mapsto I_h (\de ,D_{1\cdots m}^{\vee}) D_{1\cdots m}$ 
  is given by  $NH$. 
  By Proposition \ref{vanishing-cycle}, we have 
  \begin{align*}
    &I_h (\De_I ,D_{1\cdots m}^{\vee}) D_{1\cdots m} 
    =I_h (\De_I ,\De_I ^{\vee})D_{1\cdots m}
    =\sum _{I'} I_h (\De_I ,\De_I^{\vee}) \De_{I'} \\
    &=(\De ,\De_1 ,\De_2 ,\ldots ,\De_m ,\De_{12} ,\De_{13} ,\ldots ,\De_{1\cdots m})
    \left(
      \begin{array}{c}
        I_h (\De_I ,\De_I^{\vee}) \\ I_h (\De_I ,\De_I^{\vee}) \\ 
        \vdots \\I_h (\De_I ,\De_I^{\vee})
      \end{array}
    \right) ,
  \end{align*}
  and hence the claim is proved. 
\end{proof}

\begin{Rem}
  Let $\rho_{\infty}$ be a loop in $X$ 
  turning the hyperplane $L_{\infty} \subset \P^m$ at infinity. 
  Because of 
  $$
  \rho_{\infty} =\eta_{\vep} (\ell_1 \cdots \ell_m 
  \ell_{1\cdots 1} \ell_{1\cdots 10} \cdots \ell_{0\cdots 0})^{-1},
  $$
  we can express $\CM(\rho_{\infty})$ by 
  Corollaries \ref{CM_k}, \ref{CM_0-2}, 
  equalities (\ref{eta-rho-1}) and (\ref{eta-rho-2}); 
  see Appendix \ref{section-pi1}, for the notations $\eta_{\vep}$ and $\ell_*$. 
  However, it is too complicated to write down. 
  Here, we give the eigenvalues of $\CM(\rho_{\infty})$. 
  Similarly to Section 2.3 of \cite{Kato}, it turns out that 
  $x_m ^{-a} f(\frac{x_1}{x_m},\ldots ,\frac{x_{m-1}}{x_m},\frac{1}{x_m})$ 
  is a solution to $E_C (a,b,c)$ 
  if and only if $f(\xi_1,\ldots ,\xi_m)$ is a solution to 
  $E_C (a,a-c_m+1,(c_1,\ldots ,c_{m-1},a-b+1))$ 
  with variables $\xi_1 ,\ldots ,\xi_m$. 
  Then an argument similar to that used for Proposition \ref{k=1,,m} shows that 
  the eigenvalues of $\CM (\rho_{\infty})$ are $\al$ and $\be$. 
  Moreover, both eigenspaces are of dimension $2^{m-1}$. 
\end{Rem}

\section{Representation matrices}\label{section-matrix}
For $0\leq i \leq m$, 
the matrix representation of $\CM_i$ with respect to the basis 
$\{ \De_I \} _I$ 
is given by $M_i$ in Corollaries \ref{CM_k} and \ref{CM_0-2}. 
However, $M_0$ is too complicated to write down. 
In this section, we give another basis $\{ \De'_I \}_I$ 
of $H_m (T_{\dot{x}},u_{\dot{x}})$ and write down 
the representation matrix of $\CM_i$ with respect to this basis. 

In this and the next sections, we use the following formulas. 
\begin{Lem}\label{lemma-for-intersection}
  For a positive integer $n$ and complex numbers $\la_1 ,\ldots ,\la_n$, we have
  \begin{align}
    \label{lem-1}
    &\sum_{N\subset \{ 1,\ldots ,n \}} \prod_{l\in N} 
    \frac{\la_l}{1-\la_l}=\prod_{l=1}^n \frac{1}{1-\la_l}, 
    \quad 
    \sum_{N\subset \{ 1,\ldots ,n \}} \prod_{l\in N} 
    \frac{1}{\la_l-1}=\prod_{l=1}^n \frac{\la_l}{\la_l-1},\\
    \label{cor-of-lem-1}
    &\sum_{N\subset \{ 1,\ldots ,n \}} 
    \prod_{l\in N}(1-\la_l) \prod_{l\not\in N}\la_l  
    =\sum_{N\subset \{ 1,\ldots ,n \}} 
    (-1)^{|N|}\prod_{l\in N}(\la_l-1) \prod_{l\not\in N}\la_l =1 ,\\
    \label{cor-of-lem-2}
    &\sum_{N\subset \{ 1,\ldots ,n \}} 
    \prod_{l\in N}(\la_l -1) =\prod_{l=1}^n \la_l .
  \end{align}
\end{Lem}
\begin{proof}
  Because of 
  $$
  1+\frac{\la_l}{1-\la_l}=\frac{1}{1-\la_l} ,\quad 
  1+\frac{1}{\la_l-1}=\frac{\la_l}{\la_l-1}, 
  $$
  we obtain (\ref{lem-1}) 
  by induction on $n$. 
  The equalities (\ref{cor-of-lem-1}) and (\ref{cor-of-lem-2}) follow from 
  the first and the second ones of (\ref{lem-1}), respectively. 
\end{proof}

Let $P$ be the $2^m \times 2^m$ matrix whose $(N,I)$-entry is 
$$
\left\{
\begin{array}{ll}
  \DS \al \be \prod_{j\not\in I} \frac{\ga_j-1}{\ga_j} \cdot 
  \frac{\prod_{n\in N}\ga_n}{(\al -\prod_{n\in N}\ga_n)(\be -\prod_{n\in N}\ga_n)}
  &(N\subset I), \\
  0&(N\not\subset I), 
\end{array}
\right.
$$
and $\{ \De'_I \}_I$ be the basis of $H_m (T_{\dot{x}},u_{\dot{x}})$ 
defined as 
\begin{align*}
  (\De' ,\De'_1 ,\De'_2 ,\ldots ,\De'_m ,\De'_{12} ,\De'_{13} ,\ldots ,\De'_{1\cdots m})
  =(\De ,\De_1 ,\De_2 ,\ldots ,\De_m ,\De_{12} ,\De_{13} ,\ldots ,\De_{1\cdots m})P .
\end{align*}
Namely, $\De'_I$ is defined by 
$$
\De'_I=\al \be \prod_{j\not\in I} \frac{\ga_j-1}{\ga_j} \cdot 
\sum_{N\subset I}
\frac{\prod_{n\in N}\ga_n}{(\al -\prod_{n\in N}\ga_n)(\be -\prod_{n\in N}\ga_n)} 
\De_N .
$$
Note that $P$ is an upper triangular matrix. 
\begin{Lem}\label{basis-lem}
  We have 
  \begin{align*}
    \frac{(\al -\prod_k \ga_k)(\be -\prod_k \ga_k)}{\al \be \prod_k \ga_k} \De'_{1\cdots m} 
    +\sum_{I\subsetneq \{ 1,\ldots ,m \}} \left(  
      \frac{1}{\prod_{i\in I}\ga_i} 
      +(-1)^{m-|I|} \frac{\prod_k \ga_k}{\al \be}
    \right) \De'_I =D_{1\cdots m} .
  \end{align*}
\end{Lem}
\begin{proof}
  By the definition, the left-hand side is equal to 
  \begin{align}
    \label{basis-lem-pf}
    &\frac{(\al -\prod_k \ga_k)(\be -\prod_k \ga_k)}{\al \be \prod_k \ga_k}
    \cdot \al \be \sum_{N\subset \{ 1,\ldots ,m \}}
    \frac{\prod_{n\in N}\ga_n}{(\al -\prod_{n\in N}\ga_n)(\be -\prod_{n\in N}\ga_n)} 
    \De_N   \\
    &+\sum_{I\subsetneq \{ 1,\ldots ,m \}} \Biggl[
    \prod_{j\not\in I}(\ga_j -1) 
    \left(  
      \frac{\al \be}{\prod_k \ga_k} +(-1)^{m-|I|} \prod_{i\in I}\ga_i
    \right)   \nonumber \\
    & \quad \quad \quad \quad \quad \quad \quad 
    \cdot \sum_{N\subset I}
    \frac{\prod_{n\in N}\ga_n}{(\al -\prod_{n\in N}\ga_n)(\be -\prod_{n\in N}\ga_n)} \De_N 
    \Biggr] .  \nonumber
  \end{align}
  Clearly the coefficient of $\De_{1\cdots m}$ in (\ref{basis-lem-pf}) is $1$. 
  The coefficient of $\De_N \ (N\neq \{ 1,\ldots ,m \})$ is 
  \begin{align*}
    &\frac{\prod_{n\in N}\ga_n}{(\al -\prod_{n\in N}\ga_n)(\be -\prod_{n\in N}\ga_n)}\\
    &\quad \cdot \Biggl(
      \frac{(\al -\prod_k \ga_k)(\be -\prod_k \ga_k)}{\prod_k \ga_k} 
      +\! \! \! \! \! \sum_{\substack{I\supset N \\ I\neq \{ 1,\ldots ,m \}}}
      \prod_{j\not\in I}(\ga_j -1) 
      \Bigl( \frac{\al \be}{\prod_k \ga_k} +(-1)^{m-|I|} \prod_{i\in I}\ga_i
      \Bigr) 
    \Biggr) 
  \end{align*}
  which equals to $1$ by the equalities (\ref{cor-of-lem-1}) and (\ref{cor-of-lem-2}). 
  Therefore, by using (\ref{D-sum}), we conclude that 
  (\ref{basis-lem-pf}) is equal to 
  $$
  \sum_{I\subset \{ 1,\ldots ,m \}} \De_I =D_{1\cdots m}. 
  $$
\end{proof}
\begin{Cor}
  For $0\leq i \leq m$, let 
  $M'_i$ be 
  the representation matrix of $\CM_i$ with respect to the basis 
  $\{ \De'_{I} \}_I$. 
  Then we have 
  \begin{align*}
      M'_0=E_{2^m}-N_0, \quad 
      M'_k=M_k +N_k \  (1\leq k \leq m) ,
  \end{align*}
  where $N_i$ is defined as follows. 
  The $(I,I')$-entry of $N_0$ (resp. $N_k$) is zero, except in the case of 
  $I'=\emptyset$ (resp. $k\in I'$ and $I=I'-\{ k \}$). 
  The $(I,\emptyset)$-entry of $N_0$ is 
  $$
  \left\{
  \begin{array}{ll}
    \DS \frac{(\al-\prod_k \ga_k)(\be -\prod_k \ga_k)}{\al \be \prod_k \ga_k}& 
    (I=\{ 1,\ldots ,m \}) ,\\
    \DS \frac{1}{\prod_{i\in I}\ga_i}+(-1)^{m-|I|}\frac{\prod_k \ga_k}{\al \be} &
    ({\rm otherwise}).
  \end{array}
  \right.
  $$
  The $(I'-\{ k\} ,I')$-entry of $N'_k$ is $1$. 
\end{Cor}
In particular, $M'_k \ (1\leq k \leq m)$ is upper triangular, 
$M'_0$ is lower triangular, 
and the $(\emptyset, \emptyset)$-entry of $M'_0$ is 
$$
1-\left( 1+(-1)^m \frac{\prod \ga_k}{\al \be} \right) 
=(-1)^{m-1} \prod \ga_k \cdot \al^{-1} \be^{-1}. 
$$
\begin{proof}
  First, we evaluate $M'_0$. By Corollary \ref{CM_0-2}, 
  it is sufficient to show that the matrix representation of 
  the linear map 
  $$
  \de \mapsto \frac{(\be-1)(\al -\prod_k \ga_k)}{\al \be} 
  I_h (\de ,D_{1\cdots m}^{\vee}) D_{1\cdots m}
  $$
  is given by $N_0$. 
  By Fact \ref{solution-cycle} and 
  Proposition \ref{vanishing-cycle}, we have 
  \begin{align*}
    \frac{(\be-1)(\al -\prod_k \ga_k)}{\al \be} 
    I_h (\De'_{I'} ,D_{1\cdots m}^{\vee}) D_{1\cdots m} 
    =\left( \sum_{N\subset I'} (-1)^{|N|} \right) 
    \prod_{i\in I'}\frac{\ga_i}{\ga_i -1} \cdot D_{1\cdots m} ,
  \end{align*}
  and hence we obtain 
  \begin{align*}
    \frac{(\be-1)(\al -\prod_k \ga_k)}{\al \be} I_h (\De'_{I'} ,D_{1\cdots m}^{\vee}) D_{1\cdots m}
    =\left\{
      \begin{array}{ll}
        D_{1\cdots m} & (I'=\emptyset) , \\
        0 & ({\rm otherwise}).
      \end{array}
    \right.
  \end{align*}
  Thus Lemma \ref{basis-lem} shows the claim. 

  Next, we evaluate $M'_k \ (1\leq k \leq m)$. We have to show that 
  \begin{align*}
    \CM_k (\De'_I)=\left\{  
      \begin{array}{ll}
        \De'_I & (k\not\in I), \\
        \ga_k^{-1} \De'_I +\De'_{I-\{ k \}} &(k\in I).
      \end{array}
    \right.
  \end{align*}
  If $k\not\in I$, then the subsets $N$ of $I$ also satisfy 
  $k\not\in N$, and hence we have $\CM_k (\De_N)=\De_N$ by 
  Proposition \ref{k=1,,m}. 
  This implies that $\CM_k (\De'_I)=\De'_I$, for $k\not\in I$. 
  We assume $k\in I$. 
  For a subset $N$ of $I-\{ k \}$, we have 
  \begin{align*}
    \CM_k (\De_N)=\De_N =\left( \ga_k^{-1}+\frac{\ga_k-1}{\ga_k} \right) \De_N ,\quad 
    \CM_k (\De_{N\cup \{ k \}}) =\ga_k^{-1} \De_{N\cup \{ k \}}.
  \end{align*}
  Then we obtain 
  \begin{align*}
    \CM_k (\De'_I) 
    &=\ga_k^{-1} \De'_I 
    +\frac{\ga_k-1}{\ga_k} \cdot \al \be \prod_{j\not\in I} \frac{\ga_j-1}{\ga_j} 
    \cdot \! \! \! \! \! \sum_{N\subset I-\{ k \}}
    \frac{\prod_{n\in N}\ga_n}{(\al -\prod_{n\in N}\ga_n)(\be -\prod_{n\in N}\ga_n)} 
    \De_N \\
    &=\ga_k^{-1} \De'_I 
    +\al \be \prod_{j\not\in I-\{ k \}} \frac{\ga_j-1}{\ga_j} 
    \cdot \sum_{N\subset I-\{ k \}}
    \frac{\prod_{n\in N}\ga_n}{(\al -\prod_{n\in N}\ga_n)(\be -\prod_{n\in N}\ga_n)} 
    \De_N \\
    &=\ga_k^{-1} \De'_I +\De'_{I-\{ k \}} . 
  \end{align*}
\end{proof}

\begin{Ex}
  We write down $M'_i \ (0\leq i \leq m)$ for $m=2,3$. 
  \begin{enumerate}[(i)]
  \item In the case of $m=2$, 
    the representation matrices $M'_0, M'_1, M'_2$ are as follows: 
    \begin{align*}
      &M'_0 =\left(  
        \begin{array}{cccc}
          -\frac{\ga_1 \ga_2}{\al \be} &0&0&0 \\
          -\frac{1}{\ga_1}+\frac{\ga_1 \ga_2}{\al \be}&1&0&0 \\
          -\frac{1}{\ga_2}+\frac{\ga_1 \ga_2}{\al \be}&0&1&0 \\
          -\frac{(\al -\ga_1 \ga_2)(\be-\ga_1 \ga_2)}{\al \be \ga_1 \ga_2} &0&0&1
        \end{array}
      \right), \\
      &M'_1 =\left(  
        \begin{array}{cccc}
          1&1&0&0 \\
          0&\frac{1}{\ga_1}&0&0 \\
          0&0&1&1 \\
          0&0&0&\frac{1}{\ga_1}
        \end{array}
      \right) ,\quad 
      M'_2 =\left(  
        \begin{array}{cccc}
          1&0&1&0 \\
          0&1&0&1 \\
          0&0&\frac{1}{\ga_2}&0 \\
          0&0&0&\frac{1}{\ga_2}
        \end{array}
      \right) .
    \end{align*}
    These are equal to the transposed matrices of those in 
    Remark 4.4 of \cite{GM}. 
  \item In the case of $m=3$, 
    the representation matrices $M'_0, M'_1, M'_2 ,M'_3$ are as follows: 
    {\allowdisplaybreaks
    \begin{align*}
      &M'_0 =\left(  
        \begin{array}{cccccccc}
          \frac{\ga_1 \ga_2 \ga_3}{\al \be} &0&0&0&0&0&0&0 \\
          -\frac{1}{\ga_1}-\frac{\ga_1 \ga_2 \ga_3}{\al \be}&1&0&0&0&0&0&0 \\
          -\frac{1}{\ga_2}-\frac{\ga_1 \ga_2 \ga_3}{\al \be}&0&1&0&0&0&0&0 \\
          -\frac{1}{\ga_3}-\frac{\ga_1 \ga_2 \ga_3}{\al \be}&0&0&1&0&0&0&0 \\
          -\frac{1}{\ga_1 \ga_2}+\frac{\ga_1 \ga_2 \ga_3}{\al \be}&0&0&0&1&0&0&0 \\
          -\frac{1}{\ga_1 \ga_3}+\frac{\ga_1 \ga_2 \ga_3}{\al \be}&0&0&0&0&1&0&0 \\
          -\frac{1}{\ga_2 \ga_3}+\frac{\ga_1 \ga_2 \ga_3}{\al \be}&0&0&0&0&0&1&0 \\
          -\frac{(\al -\ga_1 \ga_2 \ga_3)(\be-\ga_1 \ga_2 \ga_3)}{\al \be \ga_1 \ga_2 \ga_3} 
          &0&0&0&0&0&0&1
        \end{array}
      \right) ,\\
      &M'_1 =\left(  
        \begin{array}{cccccccc}
          1&1&0&0&0&0&0&0 \\
          0&\frac{1}{\ga_1}&0&0&0&0&0&0 \\
          0&0&1&0&1&0&0&0 \\
          0&0&0&1&0&1&0&0 \\
          0&0&0&0&\frac{1}{\ga_1}&0&0&0 \\
          0&0&0&0&0&\frac{1}{\ga_1}&0&0 \\
          0&0&0&0&0&0&1&1 \\
          0&0&0&0&0&0&0&\frac{1}{\ga_1} 
        \end{array}
      \right) ,\quad 
      M'_2 =\left(  
        \begin{array}{cccccccc}
          1&0&1&0&0&0&0&0 \\
          0&1&0&0&1&0&0&0 \\
          0&0&\frac{1}{\ga_2}&0&0&0&0&0 \\
          0&0&0&1&0&0&1&0 \\
          0&0&0&0&\frac{1}{\ga_2}&0&0&0 \\
          0&0&0&0&0&1&0&1 \\
          0&0&0&0&0&0&\frac{1}{\ga_2}&0 \\
          0&0&0&0&0&0&0&\frac{1}{\ga_2} 
        \end{array}
      \right) ,\\ 
      &M'_3 =\left(  
        \begin{array}{cccccccc}
          1&0&0&1&0&0&0&0 \\
          0&1&0&0&0&1&0&0 \\
          0&0&1&0&0&0&1&0 \\
          0&0&0&\frac{1}{\ga_3}&0&0&0&0 \\
          0&0&0&0&1&0&0&1 \\
          0&0&0&0&0&\frac{1}{\ga_3}&0&0 \\
          0&0&0&0&0&0&\frac{1}{\ga_3}&0 \\
          0&0&0&0&0&0&0&\frac{1}{\ga_3} 
        \end{array}
      \right) .
    \end{align*}}
  \end{enumerate}
\end{Ex}

\section{Proof of the main theorem}\label{section-proof}
In this section, we prove Theorem \ref{main}. 
Since $\dim H_m (T_{\dot{x}}, u_{\dot{x}})=2^m$, 
it is sufficient to show that 
$D_I$'s are eigenvectors and linearly independent. 
First, we evaluate the intersection numbers 
$I_h (\De_I ,D_{I'}^{\vee})$. 
Second, we show the linear independence of $\{ D_I \}_I$ by 
evaluating the determinant of the matrix 
$\left( I_h (\De_I ,D_{I'}^{\vee}) \right)_{I,I'}$. 
Third, we prove the properties of 
the eigenspace of $\CM_0$ of eigenvalue $1$. 
Finally, we show that $D_{1\cdots m}$ is an eigenvector of $\CM_0$ 
of eigenvalue $(-1)^{m-1} \prod_k \ga_k \cdot \al^{-1} \be^{-1}$.

\subsection{An expression of $D_{1\cdots m}$}\label{appendix}
We prove Proposition \ref{vanishing-cycle} 
by using imaginary cycles and the $\De_I$'s introduced in Section \ref{section-solution-cycle}.  

Fix any $s_0 \in \sigma _I$, and set 
$$
\sqrt{-1}\R_I ^{m} :=
\{ s_0 +\sqrt{-1}(\eta_1 ,\ldots ,\eta_m) \mid (\eta_1 ,\ldots ,\eta_m) \in \R^m  \} 
\subset M_I , 
$$
which is called an imaginary cycle. 
By arguments similar to those in the proof of 
Proposition 4.3 and Theorem 4.4 in \cite{G-FC}, 
we can prove that the integration of $u\vph$ on $(\iota_I )_* (\sqrt{-1} \R_I ^m)$ also gives 
the solution $f_I$ to $E_C(a,b,c)$, under some conditions for the parameters $a,b,c$. 
Therefore, $(\iota_I )_* (\sqrt{-1} \R_I ^m)^{\vee}$ is orthogonal to the cycles 
$\De_{I'} \ (I' \neq I)$ with respect to $I_h$ 
(cf. Proof of Lemma 4.1 in \cite{GM}), 
and hence $(\iota_I )_* (\sqrt{-1} \R_I ^m)^{\vee}$ 
is a constant multiple of $\De_I^{\vee}$. 
Note that both $D_{1\cdots m}$ and 
$\iota_I (\si_I)$ 
intersect 
$\iota_I (\sqrt{-1} \R_I ^m)$ at $\iota_I (s_0)$ transversally. 
Since $D_{1\cdots m}$ and $\iota_I (\si_I)$ have 
a same orientation (c.f. Remark 4.5 (i) in \cite{G-FC}), 
we have
\begin{align*}
   I_h (D_{1\cdots m} ,(\iota_I )_* (\sqrt{-1} \R_I ^m)^{\vee})
  =I_h (\De_I ,(\iota_I )_* (\sqrt{-1} \R_I ^m)^{\vee}). 
\end{align*}
Thus we obtain  
$$
\De_I ^{\vee}=
\frac{I_h (\De_I ,\De_I ^{\vee})}
{I_h (D_{1\cdots m} ,(\iota_I )_* (\sqrt{-1} \R_I ^m)^{\vee})}
\cdot (\iota_I )_* (\sqrt{-1} \R_I ^m)^{\vee} , 
$$
which implies the first equality of (\ref{D-Delta}) because of 
\begin{align*}
  I_h (D_{1\cdots m} ,\De_I ^{\vee})=& 
  \frac{I_h (\De_I ,\De_I ^{\vee})}
  {I_h (D_{1\cdots m} ,(\iota_I )_* (\sqrt{-1} \R_I ^m)^{\vee})}
  \cdot I_h( D_{1\cdots m} , (\iota_I )_* (\sqrt{-1} \R_I ^m)^{\vee} ) 
  = I_h (\De_I ,\De_I ^{\vee}). 
\end{align*}
The second equality of (\ref{D-Delta}) is shown as 
$$
I_h (\De_I ,D_{1\cdots m}^{\vee}) =(-1)^m I_h (D_{1\cdots m} ,\De_I ^{\vee})^{\vee}
=(-1)^m I_h (\De_I ,\De_I ^{\vee})^{\vee} 
=I_h (\De_I ,\De_I ^{\vee}), 
$$
where $g(\al ,\be ,\ga_1 ,\ldots ,\ga_m)^{\vee} 
:=g(\al^{-1} ,\be^{-1} ,\ga_1^{-1} ,\ldots ,\ga_m^{-1})$ for 
$g(\al ,\be ,\ga_1 ,\ldots ,\ga_m) \in \C(\al ,\be ,\ga_1 ,\ldots ,\ga_m)$.  
The orthogonality of the $\De_I$'s implies 
$$
D_{1\cdots m} =\sum_{I} 
\frac{I_{h} (D_{1\cdots m},\De _{I}^{\vee})}{I_{h} (\De _{I} ,\De _{I}^{\vee})} \De _I 
=\sum _{I} \De_I ,
$$
which is the equality (\ref{D-sum}). 
Hence the self-intersection number of $D_{1\cdots m}$ is 
\begin{align*}
  & I_{h} (D_{1\cdots m} ,D_{1\cdots m} ^{\vee})
  =\sum_I I_{h} (\De _I ,\De _I ^{\vee}) \\
  &=\sum_{I} (-1) ^{|I|} \frac{\prod_{j\not\in I} \ga _j 
    \cdot \left( \al -\prod_{i\in I} \ga _i \right) \left( \be -\prod_{i\in I} \ga_i \right) }
  {\prod_k (\ga _k -1) \cdot \left( \al -\prod_k \ga _k \right) (\be -1)} 
  =\frac{\al \be +(-1)^m \prod_k \ga_k }{(\be-1)(\al-\prod_k \ga_k)}.
\end{align*}
At the last equality, we use (\ref{cor-of-lem-2}).
Therefore, Proposition \ref{vanishing-cycle} is proved.

\subsection{Intersection numbers}
For $I,I' \subset \{ 1,\ldots ,m \}$, we evaluate the intersection number 
$I_h (\De_I ,D_{I'}^{\vee})$. 
By Proposition \ref{vanishing-cycle}, we may assume 
$I' \neq \{ 1,\ldots ,m \}$. 
We set
\begin{align*}
  &J:=\{ 1,\ldots ,m \} -I, \quad 
  J':=\{ 1,\ldots ,m \} -I',\\
  &I_0:=I\cap I' ,\quad I_1:=I\cap J' ,\quad 
  J_0:=J\cap I' ,\quad J_1:=J\cap J' .
\end{align*}
By using $\iota_I$, we have 
$I_h (\De_I ,D_{I'}^{\vee})=I_h (\tilde{\De}_I ,\tilde{D}_{I'}^{\vee})$, 
where $\tilde{D}_{I'}:=(-1)^{|I|} \cdot (\iota_I)_* ^{-1} (D_{I'})$. 
Note that the orientation of $\tilde{D}_{I'}$ is also induced from the
natural embedding $\R^m \subset \C^m$. 
Thus $\si_I$ and $\tilde{D}_{I'}$ have the same orientation. 
For $I'\neq \emptyset$, 
$\tilde{D}_{I'}$ is a chamber 
$$
\left\{ (s_1 ,\ldots ,s_m) \in \R^m \ \left| 
    \begin{array}{l}
      s_i >0 \ (i\in I'),\ s_j <0 \ (j \not\in I'),\\ 
      (-1)^{|I_1|} v_I(s)>0,\ (-1)^{|I_1|+|J'|+1}w_I(s)>0 
    \end{array}
  \right. \right\} 
$$
loaded the branch of $u_I$ by the assignment of arguments 
as in the following. 
$$
\begin{array}{|c|c|c|c|c|c|}
  \hline 
  & s_i (i\in I') & s_i (i\in I_1) & s_i (i\in J_1) & v_I(s) & w_I(s) 
  \\ \hline 
  {\rm argument} & 0 & \pi & -\pi & |I_1| \pi & ( |I_1|-(|J'|+1) ) \pi 
  \\ \hline
\end{array}
$$
In fact, the conditions for $v_I$ and $w_I$ are simply given by
$$
1-\sum_{i\in I} \frac{x_i}{s_i} -\sum_{j\in J} s_j >0,\quad 
1-\sum_{i\in I} s_i -\sum_{j\in J} \frac{x_j}{s_j} <0 , 
$$
respectively, 
because of $|J'|=|I_1|+|J_1|$. 
In the case of $I' =\emptyset$ (then $I_0 =J_0=\emptyset$), 
$\tilde{D_{\emptyset}}=\tilde{D}$ is a chamber 
$$
\{ (s_1 ,\ldots ,s_m) \in \R^m \mid 
s_k <0 \ (1 \leq k \leq m) \} 
$$
loaded the branch of $u_I$ by the assignment of arguments 
as in the following. 
$$
\begin{array}{|c|c|c|c|c|}
  \hline 
  & s_i (i\in I_1) & s_i (i\in J_1) & v_I(s) & w_I(s) 
  \\ \hline 
  {\rm argument} & \pi & -\pi & |I_1| \pi & ( |I_1|-m ) \pi 
  \\ \hline
\end{array}
$$
\begin{Lem}\label{intersection-0}
  If $I'\neq \emptyset$ and $I\subset J'$, we have 
  $I_h (\tilde{\De}_I ,\tilde{D}_{I'}^{\vee})=0$. 
\end{Lem}
\begin{proof}
  By the assumption, we have $J_0 =J\cap I'=I'\neq \emptyset$. 
  For $(s_1 ,\ldots ,s_m) \in \tilde{D}_{I'}$, 
  we show that at least one of the $s_j$'s ($j\in J_0$) satisfies 
  $0<s_j <m x_j$. 
  Because of $m x_j <m\cdot \frac{\vep^2}{m}<\vep$, 
  it implies that the chamber $\tilde{D}_{I'}$ is included in 
  the $\vep$-neighborhood of $(s_j=0)$, and hence 
  $\tilde{D}_{I'}$ does not intersect $\tilde{\De}_I$. 
  Thus, the lemma is proved. 
  We assume that all of the $s_j$'s ($j\in J_0$) satisfy $s_j \geq m x_j$. 
  By 
  $$
  0>1-\sum_{i\in I} s_i -\sum_{j\in J} \frac{x_j}{s_j}
  =1-\sum_{i\in I_1} s_i -\sum_{j\in J_0} \frac{x_j}{s_j}-\sum_{j\in J_1} \frac{x_j}{s_j} , 
  $$
  $s_i <0 \ (i\in I_1)$ and $s_j <0 \ (j\in J_1)$, 
  we have 
  $$
  1<1-\sum_{i\in I_1} s_i -\sum_{j\in J_1} \frac{x_j}{s_j}<\sum_{j\in J_0} \frac{x_j}{s_j}.
  $$
  However, the inequalities
  $$
  \sum_{j\in J_0} \frac{x_j}{s_j}\leq \sum_{j\in J_0} \frac{x_j}{m x_j}
  \leq \sum_{j\in J_0} \frac{1}{m} \leq 1
  $$
  lead to a contradiction to $1<\sum_{j\in J_0}\frac{x_j}{s_j}$. 
\end{proof}
We consider in the case of $I'\neq \emptyset$. 
By Lemma \ref{intersection-0}, we may assume that $I \not\subset J'$. 
If we consider $x_1 ,\ldots ,x_m \to 0$, 
the condition $(-1)^{|I_1|} v_I(s)>0$ may be replaced with 
$1-\sum_{j\in J} s_j>0$, and 
$(-1)^{|I_1|+|J'|+1}w(s)>0$ may be replaced with 
$1-\sum_{i\in I} s_i<0$ 
to judge if $s$ belongs to a central area of $\tilde{D}_{I'}$. 
This observation means that we can evaluate the intersection number 
$I_h (\tilde{\De}_I ,\tilde{D}_{I'}^{\vee})$ like that of 
the regularization of $V_I$ and ${V'_{I'}}^{\vee}$
by omitting the difference of the branches of $u_I$, where 
\begin{align}
  V_I &:=\left\{ (s_1,\ldots ,s_m)\in \R^m \ \left| \  
  s_k>0,\ 1-\sum_{i\in I} s_i>0, \ 1-\sum_{j\in J}s_j >0 \right. \right\} ,\nonumber \\
  \label{nearly}
  V'_{I'}&:=\left\{ (s_1,\ldots ,s_m)\in \R^m \left| 
      \begin{array}{c}
        s_k>0 \ (k\in I'),\ s_k<0 \ (k\in J'),\\ 
        1-\sum_{i\in I} s_i<0, \ 1-\sum_{j\in J}s_j >0
      \end{array}
    \right. \right\} .
\end{align}
Note that the chamber $V'_{I'}$ is not empty, because of $I \not\subset J'$. 
In the case of $I'=\emptyset$, we can see that 
the above claim is valid, by replacing (\ref{nearly}) with 
$$
V':= \{ (s_1,\ldots ,s_m)\in \R^m \mid 
s_k<0 \ (1\leq k \leq m) \} 
$$
(note that $1-\sum_{i\in I} s_i>0$ and $1-\sum_{j\in J}s_j >0$ hold clearly).
Recall that 
when we construct the twisted cycle $\tilde{\De}_I$, 
the exponents of $(s_i=0)$, $(s_j=0)$, 
$(1-\sum_{i\in I} s_i =0)$ and $(1-\sum_{j\in J} s_j=0)$ are 
$$
c_i-1,\quad 1-c_j ,\quad -b ,\quad \sum_{k=1}^m c_k -a-m+1 ,
$$
respectively, where $i\in I$ and $j\in J$; see Section 4 of \cite{G-FC}.  

\begin{Th}\label{intersection-theorem}
  For $I'\neq \emptyset$, we have 
  \begin{align}    
    \label{intersection2-nonempty}
    I_h (\tilde{\De}_I ,\tilde{D}_{I'}^{\vee}) 
    = (-1)^{m-|J_1|-1}
    \cdot \prod_{k\in J'} \frac{1}{1-\ga_k} \cdot \frac{1}{1-\be} 
    \cdot \Biggl[ 1+\sum_{\substack{K_I \subsetneq I_0 \\ K_J \subset J_0}} 
    \left( \prod_{i\in K_I}\frac{1}{\ga_i -1} \cdot \prod_{j\in K_J} \frac{\ga_j}{1-\ga_j} \right) \\
    \nonumber 
    \quad 
    +\frac{\al}{\prod_k \ga_k -\al} 
    \sum_{\substack{K_I \subsetneq I_0 \\ K_J \subsetneq J_0}}
    \left( \prod_{i\in K_I}\frac{1}{\ga_i -1} \cdot \prod_{j\in K_J} \frac{\ga_j}{1-\ga_j} \right) 
    \Biggr] . 
  \end{align}
  For $I' =\emptyset$, we have 
  \begin{align}
    \label{intersection2-empty}
    I_h (\tilde{\De}_I ,\tilde{D}^{\vee}) 
    =(-1)^{|I|} \cdot \prod_{k=1}^m \frac{1}{1-\ga_k}.
  \end{align}
\end{Th}
\begin{proof}
Let $s_0$ be an intersection point of $\tilde{\De}_I$ and $\tilde{D}_{I'}$. 
We denote the difference of the branches of $u_I$ at $s_0$ by 
$\chi_{I,I'}$, namely, 
$$
\chi_{I,I'}:=\frac
{{\rm the \ value}\ u_I(s_0) \ {\rm with \ respect \ to \ the \ branch \ defined \ on}\ \tilde{\De}_I}
{{\rm the \ value}\ u_I(s_0) \ {\rm with \ respect \ to \ the \ branch \ defined \ on}\ \tilde{D}_{I'}} .
$$
Note that $\chi_{I,I'}$ is independent of the choice of the intersection point $s_0$. 
We prove the theorem by two steps. 

Step 1: We show that
  \begin{align}
    \label{intersection1-nonempty}
    I_h (\tilde{\De}_I ,\tilde{D}_{I'}^{\vee}) 
    =&\ \chi_{I,I'} \cdot (-1)^{m-(|J'|+1)} \cdot \prod_{i\in I_1} \frac{1}{\ga_i -1} 
    \cdot \prod_{j\in J_1} \frac{1}{\ga_j^{-1} -1} \cdot \frac{1}{\be^{-1}-1} \\
    \nonumber
    &\cdot \Biggl[ 1+\sum_{\substack{K_I \subsetneq I_0 \\ K_J \subset J_0}} 
    \left( \prod_{i\in K_I}\frac{1}{\ga_i -1} \cdot \prod_{j\in K_J} \frac{1}{\ga_j^{-1}-1} \right) \\
    \nonumber 
    &\quad +\frac{1}{\al^{-1} \prod_k \ga_k -1} 
    \sum_{\substack{K_I \subsetneq I_0 \\ K_J \subsetneq J_0}}
    \left( \prod_{i\in K_I}\frac{1}{\ga_i -1} \cdot \prod_{j\in K_J} \frac{1}{\ga_j^{-1}-1} \right) 
    \Biggr] \quad (I'\neq \emptyset),  
  \end{align}
  \begin{align}
    I_h (\tilde{\De}_I ,\tilde{D}^{\vee}) 
    =&\ \chi_{I,\emptyset} \cdot (-1)^{m-m} \cdot 
    \prod_{i\in I}\frac{1}{\ga_i -1} \cdot \prod_{j\in J} \frac{1}{\ga_j^{-1}-1} .
    \label{intersection1-empty}
  \end{align}
  We prove (\ref{intersection1-nonempty}), by using results in \cite{KY}. 
  Obviously, we have 
  $$
  \overline{V_I} \cap \overline{V'_{I'}}
  =\left\{ (s_1,\ldots ,s_m)\in \R^m \left| 
      \begin{array}{c}
        s_j=0\ (j\in J') ,\ 1-\sum_{i\in I}s_i =0,\\ 
        s_i\geq 0\ (i\in I') ,\ 1-\sum_{j\in J} s_j \geq 0
      \end{array}
    \right. \right\} ,
  $$
  which implies that the intersection number 
  $I_h (\tilde{\De}_I ,\tilde{D}_{I'}^{\vee})$ is equal to 
  the product of 
  \begin{align*}
    \chi_{I,I'}
    \cdot \prod_{i\in I\cap J'} \frac{1}{\ga_i -1} 
    \cdot \prod_{j\in J\cap J'} \frac{1}{\ga_j^{-1} -1} \cdot \frac{1}{\be^{-1}-1}
  \end{align*}
  and the self-intersection number of the twisted cycle determined by the chamber 
  $$
  \left\{ (s_1,\ldots ,s_m)\in \R^m \left| 
      \begin{array}{c}
        s_j=0\ (j\in J') ,\ 1-\sum_{i\in I}s_i =0,\\ 
        s_i > 0\ (i\in I') ,\ 1-\sum_{j\in J} s_j > 0
      \end{array}
    \right. \right\}
  $$
  in the $(m-(|J'|+1))$-dimensional space 
  $L:=\bigcap_{j\in J'}(s_j =0) \cap (1-\sum_{i\in I}s_i =0)$. 
  To evaluate this self-intersection number,  
  we investigate the non-empty intersections of 
  $(s_i=0)\ (i\in I')$, $(1-\sum_{j\in J} s_j=0)$ with $L$. 
  \begin{enumerate}[(i)]
  \item Without $(1-\sum_{j\in J} s_j=0)$: 
    we choose subsets $K$ of $I'$ such that 
    $\bigcap_{k\in K} (s_k=0) \cap L\neq \emptyset$. 
    By the condition $1-\sum_{i\in I}s_i =0$, we have
    $$
    \bigcap_{k\in K} (s_k=0) \cap L\neq \emptyset
    \Leftrightarrow
    K\cap I\subsetneq I
    \Leftrightarrow
    K=K_I \cup K_J \ (K_I \subsetneq I,\ K_J \subset J). 
    $$
  \item With $(1-\sum_{j\in J} s_j=0)$: 
    we choose subsets $K$ of $I'$ such that 
    $\bigcap_{k\in K} (s_k=0) \cap (1-\sum_{j\in J} s_j=0)\cap  L\neq \emptyset$. 
    By the conditions $1-\sum_{i\in I}s_i =0$ and $1-\sum_{j\in J} s_j=0$, 
    we have
    \begin{align*}    
      &\bigcap_{k\in K} (s_k=0) \cap \Bigl( 1-\sum_{j\in J} s_j=0 \Bigr) \cap L\neq \emptyset \\
      &\Leftrightarrow
      K\cap I\subsetneq I,\ K\cap J \subsetneq J
      \Leftrightarrow
      K=K_I \cup K_J \ (K_I \subsetneq I,\ K_J \subsetneq J). 
    \end{align*}
  \end{enumerate}
  Therefore, the self-intersection number is equal to 
  \begin{align*}
    (-1)^{m-(|J'|+1)} \cdot&\Biggl[ 1+\sum_{\substack{K_I \subsetneq I_0 \\ K_J \subset J_0}} 
    \left( \prod_{i\in K_I}\frac{1}{\ga_i -1} \cdot \prod_{j\in K_J} \frac{1}{\ga_j^{-1}-1} \right) 
    \nonumber \\
    &\quad +\frac{1}{\al^{-1} \prod_k \ga_k -1} 
    \sum_{\substack{K_I \subsetneq I_0 \\ K_J \subsetneq J_0}}
    \left( \prod_{i\in K_I}\frac{1}{\ga_i -1} \cdot \prod_{j\in K_J} \frac{1}{\ga_j^{-1}-1} \right) 
    \Biggr] ,
  \end{align*}
  and hence (\ref{intersection1-nonempty}) is proved. 
  We can obtain 
  the equality (\ref{intersection1-empty}) in a similar way. 

  Step 2: We evaluate $\chi_{I,I'}$. 
  We consider the differences of the branches of the factors of $u_I$ at an intersection point 
  of $\tilde{\De}_I$ and $\tilde{D}_{I'}$.
  \begin{enumerate}[(i)]
  \item The argument of $s_k$ on $\tilde{\De}_I$ and $\tilde{D}_{I'}$ are 
    given as in the following. 
    $$
    \begin{array}{|c|c|c|c|}
      \hline 
      & k\in I'=I_0 \cup J_0 & k\in I_1 & k\in J_1  
      \\ \hline 
      \tilde{\De}_I & 0 & \pi & \pi  
      \\ \hline
      \tilde{D}_{I'} & 0 & \pi & -\pi  
      \\ \hline
    \end{array}
    $$
    Since the exponent of $s_j \ (j\in J)$ is $C_j =1-c_j +b$, 
    the contribution by the branch of $\prod_k s_k^{C_k}$ is 
    $\prod_{j\in J_1} (\ga_j^{-1} \be)$. 
  \item We have 
    $$
    v_I=\prod_{i\in I}s_i \cdot 
    \left( 1-\sum_{j\in J}s_j -\sum_{i\in I}\frac{x_i}{s_i} \right) ,
    $$
    and the term $\sum_{i\in I}\frac{x_i}{s_i}$ does not concern 
    the difference of the branches. 
    By (i) and the fact that $s\in V'_{I'}$ satisfies $1-\sum_{j\in J} s_j>0$, 
    both the argument of $v_I$ on $\tilde{\De}_I$ and that on $\tilde{D}_{I'}$ 
    are $|I_1| \pi$, and hence 
    the contribution by the branch of $v_I^A$ is $1$. 
  \item We have 
    $$
    w_I=\prod_{j\in J}s_j \cdot 
    \left( 1-\sum_{i\in I}s_i -\sum_{j\in J}\frac{x_j}{s_j} \right) ,
    $$
    and the term $\sum_{j\in J}\frac{x_j}{s_j}$ does not concern 
    the difference of the branches. 
    By (i) and the fact that $s\in V'_{I'}$ satisfies 
    $$
    \left\{
      \begin{array}{ll}
        1-\sum_{i\in I}s_i<0 & (I'\neq \emptyset), \\
        1-\sum_{i\in I}s_i>0 & (I'= \emptyset),
      \end{array}
    \right. 
    $$
    the arguments of $w_I$ on $\tilde{\De}_I$ and $\tilde{D}_{I'}$ 
    at the intersection points 
    are as follows: 
    \begin{align*}
      &({\rm argument\ on\ }\tilde{\De}_I)=\left\{
        \begin{array}{ll}
          (|J_1|+1)\pi &(I'\neq \emptyset), \\
          |J_1|\pi &(I'= \emptyset),
        \end{array}
      \right. \\
      &({\rm argument\ on\ }\tilde{D}_{I'})=\left\{
        \begin{array}{ll}
          (|I_1|-|J'|-1)\pi & (I'\neq \emptyset), \\
          (|I_1|-m)\pi =-|J_1|\pi  & (I'= \emptyset).
        \end{array}        
      \right. 
    \end{align*}
    Here, note that $m=|J'|=|I_1|+|J_1|$, if $I'= \emptyset$. 
    Because of $|J'|=|I_1|+|J_1|$, we obtain 
    \begin{align*}
      &({\rm difference \ of \ the \ arguments \ of \ }w_I) \\
      &=\left\{
        \begin{array}{ll}
          (|J_1|+1)\pi-(|I_1|-|J'|-1)\pi
          =2(|J_1|+1)\pi & (I'\neq \emptyset),\\
          |J_1|\pi-(-|J_1|)\pi =2|J_1|\pi  & (I'= \emptyset).
        \end{array}        
      \right. 
    \end{align*}
    Since the exponent of $w_I$ is $B=-b$, 
    the contribution by the branch of $w_I^B$ is 
    \begin{align*}
      \left\{
        \begin{array}{ll}
          \be^{-(|J_1|+1)} & (I'\neq \emptyset),\\
          \be^{-|J_1|} & (I'= \emptyset).
        \end{array}        
      \right. 
    \end{align*}
  \end{enumerate}
  We thus have 
  \begin{align*}
    \chi_{I,I'} =\prod_{j\in J_1} (\ga_j^{-1} \be) \cdot \be^{-(|J_1|+1)} 
    \ \ (I'\neq \emptyset),\quad 
    \chi_{I,\emptyset}=\prod_{j\in J_1} (\ga_j^{-1} \be) \cdot \be^{-|J_1|}.
  \end{align*}
  By Step 1, we obtain 
  (\ref{intersection2-nonempty}) and (\ref{intersection2-empty}). 
\end{proof}

To simplify the equality (\ref{intersection2-nonempty}), 
we use Lemma \ref{lemma-for-intersection}. 
We summarize the results in this subsection. 
\begin{Cor}\label{intersection-cor}
  If $I'\neq \emptyset, \{ 1,\ldots ,m \}$, then we have 
  \begin{align}
    \label{intersection-nonempty}
    I_h (\De_I ,D_{I'}^{\vee}) 
    =(-1)^{|I|+|I'|-1} \cdot \prod_{k=1}^m \frac{1}{1-\ga_k}
    \cdot \frac{\prod_{i\in I_0}\ga_i -1}{1-\be}
    \cdot \frac{\prod_k \ga_k -\al \prod_{j\in J_0}\ga_j}{\prod_k \ga_k -\al}. 
  \end{align}
  This equality holds, even if $I\subset J'$. 
  For $I'=\emptyset$, we have 
  \begin{align}
    \label{intersection-empty}
    I_h (\De_I ,D^{\vee}) 
    =(-1)^{|I|} \cdot \prod_{k=1}^m \frac{1}{1-\ga_k}.
  \end{align}
\end{Cor}
\begin{proof}
  Recall that $I_h (\De_I ,D_{I'}^{\vee})=I_h (\tilde{\De}_I ,\tilde{D}_{I'}^{\vee})$. 
  The equality (\ref{intersection-empty}) coincides with  
  that in Theorem \ref{intersection-theorem}. 
  If $I\subset J'$, then we have $I_0=I\cap I'=\emptyset$, and hence 
  $\prod_{i\in I_0}\ga_i -1=0$. Thus the right-hand side of 
  (\ref{intersection-nonempty}) is $0$, which is compatible with Lemma \ref{intersection-0}. 
  Then we have to show that the right-hand side of (\ref{intersection2-nonempty}) is
  equal to that of (\ref{intersection-nonempty}). 
  By (\ref{lem-1}), 
  we have 
  \begin{align*}
    &1+\sum_{\substack{K_I \subsetneq I_0 \\ K_J \subset J_0}} 
    \left( \prod_{i\in K_I}\frac{1}{\ga_i -1} \cdot \prod_{j\in K_J} \frac{\ga_j}{1-\ga_j} \right) 
    =(-1)^{|I_0|} \cdot \left( \prod_{i\in I_0}\ga_i -1 \right) \cdot 
    \prod_{k\in I'} \frac{1}{1-\ga_k}, 
  \end{align*}
  \begin{align*}
    &\sum_{\substack{K_I \subsetneq I_0 \\ K_J \subsetneq J_0}}
    \left( \prod_{i\in K_I}\frac{1}{\ga_i -1} \cdot \prod_{j\in K_J} \frac{\ga_j}{1-\ga_j} \right) 
    =(-1)^{|I_0|} \cdot \left( \prod_{i\in I_0}\ga_i -1 \right) \cdot 
    \left( 1-\prod_{j\in J_0}\ga_j \right) \cdot \prod_{k\in I'} \frac{1}{1-\ga_k} .
  \end{align*}
  Therefore, we obtain 
  \begin{align*}
    &I_h (\De_I ,D_{I'}^{\vee}) =I_h (\tilde{\De}_I ,\tilde{D}_{I'}^{\vee}) \\
    &= (-1)^{m-|J_1|-1}
    \cdot \prod_{k\in J'} \frac{1}{1-\ga_k} \cdot \frac{1}{1-\be} \\
    &\quad \cdot (-1)^{|I_0|} \cdot \left( \prod_{i\in I_0}\ga_i -1 \right) \cdot 
    \prod_{k\in I'} \frac{1}{1-\ga_k}\cdot 
    \left( 1+\frac{\al}{\prod_k \ga_k -\al}\cdot 
      \left( 1-\prod_{j\in J_0}\ga_j \right) \right) \\
    &=(-1)^{|I_1|+|J_0|-1} \cdot \prod_{k=1}^m \frac{1}{1-\ga_k}
    \cdot \frac{\prod_{i\in I_0}\ga_i -1}{1-\be}
    \cdot \frac{\prod_k \ga_k -\al \prod_{j\in J_0}\ga_j}{\prod_k \ga_k -\al}. 
  \end{align*}
  Here we use $m=|I_0|+|I_1|+|J_0|+|J_1|$. Further, since 
  $$
  |I_1|+|J_0|=|I\cap {I'}^c|+|I^c \cap I'|
  =|I\cup I'|-|I\cap I'|=|I|+|I'|-2|I\cap I'|, 
  $$
  we have $(-1)^{|I_1|+|J_0|-1}=(-1)^{|I|+|I'|-1}$. 
\end{proof}

\begin{Lem}\label{orthogonal}
  We have $I_h(D_{1\cdots m} ,D_{I'}^{\vee})=0$, if $I'\neq \{ 1,\ldots ,m \}$.  
\end{Lem}
\begin{proof}
  This is obvious, since 
  \begin{align*}
    &\overline{D_{1\cdots m}}\subset \{ (s_1 ,\ldots ,s_m) \in \R^m 
    \mid s_k>x_k \ (1\leq k \leq m) \} , \\
    &\overline{D_{I'}}\cap \{ (s_1 ,\ldots ,s_m) \in \R^m 
    \mid s_k\geq x_k \ (1\leq k \leq m) \} =\emptyset .
  \end{align*}
\end{proof}

\subsection{Linear independence}
Let $\La_0$ be the matrix $\left( I_h (\De_I,D_{I'}) \right)_{I,I'}$ 
with arranging $I,\ I'$ in the same way as 
the basis $\{ \De_I \}_{I}$ (see Section \ref{section-THG}). 
In this subsection, we evaluate the determinant of $\La_0$. 
\begin{Th}\label{lin-indep}
  We have 
    \begin{align*}
    &\det \La_0 \\ &=\left\{
      \begin{array}{ll}
        \DS
        -\left( \al \be -\prod_{k=1}^m \ga_k \right)
        \frac{(\prod_k \ga_k +\al)^{2^{m-1}-1}}
        {(1-\be)^{2^m-1}(\prod_k \ga_k -\al)^{2^{m-1}}}
        \cdot \prod_{k=1}^m \frac{1}{(1-\ga_k)^{2^{m-1}}} 
        &(m\ {\rm :odd}), \\
        \DS
        \left( \al \be + \prod_{k=1}^m \ga_k \right)
        \frac{(\prod_k \ga_k +\al)^{2^{m-1}-2}}
        {(1-\be)^{2^m-1}(\prod_k \ga_k -\al)^{2^{m-1}-1}}
        \cdot \prod_{k=1}^m \frac{1}{(1-\ga_k)^{2^{m-1}}} 
        &(m \ {\rm :even}).
      \end{array}
      \right. 
  \end{align*}
  In particular, we obtain $\det \La_0 \neq 0$, 
  and hence $\{ D_I \}_I$ is linearly independent. 
\end{Th}
\begin{Rem}\label{generic}
  In this paper, we assume that the parameters 
  $a,\ b$, and $c=(c_1 ,\ldots ,c_m)$ are generic. 
  In fact, it is sufficient for our proof of Theorem \ref{main} to assume 
  the irreducibility condition of the system $E_C(a,b,c)$ 
  \begin{align*}
    a-\sum _{i\in I} c_i ,\quad  b-\sum _{i\in I} c_i \not\in \Z 
    \quad ( I \subset \{ 1,\ldots ,m \} ), 
  \end{align*}
  and the conditions
  \begin{align*}
    c_1 ,\ldots ,c_m \not\in \Z,\quad 
    a-\sum_{k=1}^m c_k \not\in \frac{1}{2}\Z , \quad 
    a+b-\sum_{k=1}^m c_k +\frac{m+1}{2} \not \in \Z .
  \end{align*}
\end{Rem}
To compute $\det \La_0$, 
we change $\La_0$ by elementary transformations with keeping the 
determinant 
as follows. 
Add the first, second, $\ldots$, $(2^m -1)$-th row of $\La_0$
to the $2^m$-th row of $\La_0$, 
then $2^m$-th row becomes 
\begin{align*}
  &\left( I_h \Bigl( \sum_I \De_I,D^{\vee} \Bigr) ,\ldots ,
    I_h \Bigl( \sum_I \De_I,D_{2\cdots m}^{\vee} \Bigr) ,
    I_h \Bigl( \sum_I \De_I,D_{1\cdots m}^{\vee} \Bigr) \right) \\
  &=\left( I_h(D_{1\cdots m},D^{\vee}) ,\ldots ,I_h (D_{1\cdots m},D_{2\cdots m}^{\vee}),
    I_h (D_{1\cdots m},D_{1\cdots m}^{\vee}) \right) \\
  &=\left( 0,\ldots ,0, I_h (D_{1\cdots m},D_{1\cdots m}^{\vee}) \right) ,
\end{align*}
by Lemma \ref{orthogonal}. 
It means that 
$$
\det \La_0 =I_h (D_{1\cdots m},D_{1\cdots m}^{\vee}) \cdot \det \La' ,
$$
where $\La'$ is the leading principal minor of $\La_0$ of size $2^m-1$. 
By Proposition \ref{vanishing-cycle} and Corollary \ref{intersection-cor}, we have 
\begin{align*}
  \det \La_0 
  =\frac{\al \be +(-1)^m \prod_k \ga_k}
  {(1-\be)^{2^m-1}(\prod_k \ga_k -\al)^{2^m-1}}
  \cdot \prod_{k=1}^m \frac{1}{(1-\ga_k)^{2^m-1}} \cdot \det \La ,
\end{align*}
where $\La$ is a $(2^m-1)\times (2^m-1)$ matrix whose $(I,I')$-entry is 
\begin{align*}
  &\La_{I,I'}:=
  (-1)^{|I|+|I'|-1} \cdot \left( \prod_{i\in I\cap I'}\ga_i -1 \right)
  \cdot \left( \prod_{k=1}^m \ga_k -\al \prod_{j\in I^c \cap I'}\ga_j \right) \quad 
  ( I'\neq \emptyset),\\
  &\La_{I,\emptyset} :=(-1)^{|I|}. 
\end{align*}
We write 
$$
\La =\left(  
  \begin{array}{cccc}
    \La (0,0) & \La (0,1) & \cdots & \La (0,m-1) \\
    \La (1,0) & \La (1,1) & \cdots & \La (1,m-1) \\
    \vdots & \vdots & \ddots & \vdots \\
    \La (m-1,0) & \La (m-1,1) & \cdots & \La (m-1,m-1) 
  \end{array}
\right) ,
$$
where $\La (k,k')$ is the ${m \choose k} \times {m \choose k'}$ matrix. 
Note that the entries of $\La (k,k')$ are the $(I,I')$-entries of $\La$ with 
$|I|=k,\ |I'|=k'$. 

We compute $\det \La$. 
Put $\La^{(0)}:=\La$. 
We take $\La^{(n)}$ by induction on $n$ as follows: 
for $n\geq 1$, we define $\La^{(n)}$ by replacing the columns of $I'\ (|I'|\geq n+1)$ 
of $\La^{(n-1)}$ with 
\begin{align*}
  \La^{(n-1)}_{*,I'} +\sum_{\substack{K'\subset I'\\ |K'|=n}} (-1)^{|I'|+n+1} 
  \frac{\prod_k \ga_k +(-1)^n \al \prod_{j\in {K'}^c \cap I'} \ga_j}
  {\prod_k \ga_k +(-1)^n \al} \cdot \La^{(n-1)}_{*,K'}, 
\end{align*}
where $\La^{(n-1)}_{*,I'}$ is the column of $I'$ of $\La^{(n-1)}$. 
Straightforward calculations show the following lemma. 

\begin{Lem}\label{transform}
  \begin{enumerate}[{\rm (i)}]
  \item $\det \La^{(n)} =\det \La$, $\La^{(n)}_{\emptyset ,\emptyset}=1$,  
  \item if $|I'|\geq n+1$, then 
    \begin{align*}
      \La^{(n)}_{I,I'}  
      =&(-1)^{|I|+|I'|-1} 
      \cdot \Biggl[ 
      \left( \prod_{i\in I\cap I'}\ga_i -1 \right)
      \cdot \left( \prod_{k=1}^m \ga_k -\al \prod_{j\in I^c \cap I'}\ga_j \right) \\
      &\quad \quad -\sum_{\substack{K \subset I\cap I' \\ 0<|K|\leq n}}
      \left( \prod_{i\in K}(\ga_i -1) \cdot \left( \prod_{k=1}^m \ga_k 
        +(-1)^{|K|} \al \prod_{j\in K^c \cap I'} \ga_j \right) \right)
      \Biggr] ,
    \end{align*}
  \item $k\leq n \Longrightarrow \La^{(n)}(k,k')=O \ (k'> k)$, 
  \item $\La^{(n)}(1,1),\ldots ,\La^{(n)}(n+1,n+1)$ are diagonal, 
  \item $1\leq |I|\leq n+1 \Longrightarrow 
    \La^{(n)}_{I,I}=-\prod_{i\in I}(\ga_i -1) \cdot \left( \prod_k \ga_k +(-1)^{|I|} \al \right)$. 
  \end{enumerate}
\end{Lem}
Note that the columns of $I'\ (|I'|\leq n)$ and 
the rows of $I\ (|I|\leq n-1)$ are equal to those of $\La^{(n-1)}$. 
By using this lemma, we prove Theorem \ref{lin-indep}. 
\begin{proof}[Proof of Theorem \ref{lin-indep}]
By Lemma \ref{transform}, $\La^{(m-2)}$ is the lower triangular matrix 
whose diagonal entries are given by (i) and (v). 
Hence we obtain 
\begin{align*}
  &\det \La_0 
  =\frac{\al \be +(-1)^m \prod_k \ga_k}
  {(1-\be)^{2^m-1}(\prod_k \ga_k -\al)^{2^m-1}}
  \cdot \prod_{k=1}^m \frac{1}{(1-\ga_k)^{2^m-1}} \cdot \det \La^{(m-2)} \\
  &=(-1)^m \cdot \frac{\al \be +(-1)^m \prod_k \ga_k}
  {(1-\be)^{2^m-1}(\prod_k \ga_k -\al)^{2^m-1}}
  \cdot \prod_{k=1}^m \frac{1}{(1-\ga_k)^{2^{m-1}}} 
  \cdot \prod_{\emptyset \neq I \subsetneq \{ 1,\ldots ,m \}}
  \left( \prod_{k=1}^m \ga_k +(-1)^{|I|} \al \right) .
\end{align*}
If $m$ is odd, we have 
$$
\prod_{\emptyset \neq I \subsetneq \{ 1,\ldots ,m \}}
\left( \prod_{k=1}^m \ga_k +(-1)^{|I|} \al \right)
=\left( \prod_{k=1}^m \ga_k - \al \right) ^{2^{m-1}-1}
\cdot \left( \prod_{k=1}^m \ga_k + \al \right) ^{2^{m-1}-1} .
$$
If $m$ is even, we have
$$
\prod_{\emptyset \neq I \subsetneq \{ 1,\ldots ,m \}}
\left( \prod_{k=1}^m \ga_k +(-1)^{|I|} \al \right)
=\left( \prod_{k=1}^m \ga_k - \al \right) ^{2^{m-1}}
\cdot \left( \prod_{k=1}^m \ga_k + \al \right) ^{2^{m-1}-2} .
$$
Therefore, 
the proof of Theorem \ref{lin-indep} is completed. 
\end{proof}

\subsection{The eigenspace of $\CM_0$ associated to $1$}
By Lemma \ref{orthogonal} and Theorem \ref{lin-indep}, we have to show that
\begin{itemize}
\item[$\bullet$] $\CM_0 (D_I)=D_I$ for $I\subsetneq \{ 1,\ldots ,m \}$,
\item[$\bullet$] $\CM_0 (D_{1\cdots m})=
  \left[ (-1)^{m-1} \prod_k \ga_k \cdot \al^{-1} \be^{-1} \right] \cdot D_{1\cdots m}$, 
\end{itemize}
to prove Theorem \ref{main}. 
In this subsection, we show the first claim. 
The second one is proved in the next subsection. 

Hereafter, we use the coordinates 
$\DS (s_1 ,\ldots s_m)=\left( \frac{t_1}{x_1},\ldots ,\frac{t_m}{x_m} \right)$. 
The functions 
$v(t)$ and $w(t,x)$ are expressed as
$$
1-\sum_{k=1}^m x_k s_k ,\quad 
\prod_{k=1}^m (x_k s_k)  \cdot \left( 1-\sum_{k=1}^m \frac{1}{s_k} \right) ,
$$
respectively. 
Let 
$$
v'(s,x):=1-\sum_{k=1}^m x_k s_k  ,\quad 
w'(s):=\prod_{k=1}^m s_k  \cdot \left( 1-\sum_{k=1}^m \frac{1}{s_k} \right) .
$$
If $x_1 ,\ldots, x_m$ are positive real numbers, 
then we have 
\begin{align*}
  t_k \gtreqqless 0 \Leftrightarrow s_k \gtreqqless 0 ,\quad 
  v(t) \gtreqqless 0 \Leftrightarrow v'(s,x) \gtreqqless 0 ,\quad 
  w(t,x) \gtreqqless 0 \Leftrightarrow w'(s) \gtreqqless 0 , 
\end{align*}
and hence the expressions of the $D_I$'s are as follows:
\begin{align*}
  D_{1\cdots m} : &\ s_k >0 \ (1\leq k \leq m),\ v'(s,x)>0,\ w'(s)>0, \\
  D : &\ s_k <0 \ (1\leq k \leq m),\\
  D_I \ ({\rm otherwise}) : &\ 
  s_i>0 \ (i\in I),\ s_j <0\ (j\not\in I),\ v'(s,x)>0,\ (-1)^{m-|I|+1} w'(s)>0 .
\end{align*}
Note that if $x=(x_1,\ldots ,x_m)$ moves, then only the divisor $(v'(s,x)=0)$ varies. 

Recall that the loop $\rho_0$ 
is homotopic to the composition $\tau_0 \rho'_0 \overline{\tau_0}$, 
where 
\begin{align*}
  \tau_0 &:[0,1] \ni \theta \mapsto 
  \left( (1-\theta) \cdot \frac{1}{2m^2}+\theta \cdot \left( \frac{1}{m^2}-\vep_0 \right) \right) 
  (1,\ldots ,1) \in X, \\
  \rho'_0 &:[0,1] \ni \theta \mapsto 
  \left( \frac{1}{m^2} -\vep_0 e^{2\pi \sqrt{-1}\theta}  \right)
  (1,\ldots ,1) \in X,
\end{align*}
for a sufficiently small positive real number $\vep_0$. 
Since variations along the paths $\tau_0$ and $\overline{\tau_0}$ give 
trivial transformations of the cycles $D_I$'s, 
we have to consider the variation along $\rho'_0$ for a sufficiently small $\vep_0$. 
Let 
$x \to \left( \frac{1}{m^2},\ldots ,\frac{1}{m^2} \right)$, 
then $(v'(s,x)=0)$ and $(w'(s)=0)$ are tangent 
at $(s_1 ,\ldots ,s_m)=(m,\ldots ,m)$. 
Thus $D_{1\cdots m}$ is a vanishing cycle. 
Each $D_I \ (I\subsetneq \{ 1,\ldots ,m \})$ survives as 
$x \to \left( \frac{1}{m^2},\ldots ,\frac{1}{m^2} \right)$, 
and its variation along $\rho'_0$ is too slight to change the branch of $u_x$ on it. 
This implies that $\CM_0 (D_I) =D_I$ for $I\subsetneq \{ 1,\ldots ,m \}$.

\subsection{An eigenvector of $\CM_0$ associated to the eigenvalue 
$(-1)^{m-1}  \prod_k \ga_k \cdot \al^{-1} \be^{-1}$}
In this subsection, we show $\CM_0 (D_{1\cdots m})=
[(-1)^{m-1} \prod_k \ga_k \cdot \al^{-1} \be^{-1} ]\cdot D_{1\cdots m}$. 
As mentioned in the previous subsection, 
it is sufficient to consider the variation of $D_{1\cdots m}$ along 
$\rho'_0$ for a sufficiently small $\vep_0$. 
Thus we may consider that $D_{1\cdots m}$ is contained in a small 
neighborhood of $s=(m,\ldots ,m)$ in $\R^m$. 

Putting $x_1=\cdots =x_m=\frac{1}{m^2} -\vep_0$, we have
$$
v'(s,\rho'_0(0))
=1-\left( \frac{1}{m^2} -\vep_0  \right) 
\sum_{k=1}^m s_k .
$$
We use the coordinates system 
$$
(s'_1,\ldots ,s'_{m-1},s'_m):=
\left( s_1 -m,\ldots ,s_{m-1}-m ,\sum_{k=1}^m s_k -m^2 \right) .
$$
Note that $s_l=s'_l+m \ (1\leq l \leq m-1)$ and 
$s_m=s'_m -\sum_{l=1}^{m-1} s'_l +m$. 
Then the origin $(s'_1 ,\ldots ,s'_m)=(0,\ldots ,0)$ corresponds to 
$(s_1,\ldots, s_m)=(m\ldots ,m)$. 
Let $U$ be a small neighborhood of $(s'_1 ,\ldots ,s'_m)=(0,\ldots ,0)$ 
so that $s_k >0 \ (1\leq k \leq m)$. 
In $U$, we have 
\begin{align*}
  v'(s,\rho'_0(0))>0 \Leftrightarrow &\ 
  1-\left( \frac{1}{m^2} -\vep_0  \right) (s'_m +m^2)>0 
  \Leftrightarrow  
  s'_m <\frac{m^2}{\frac{1}{m^2}-\vep_0} \cdot \vep_0 ,\\
  w'(s)>0 \Leftrightarrow &\ 
  1-\sum_{k=1}^m \frac{1}{s_k}>0 
  \Leftrightarrow 
  s'_m>\sum_{l=1}^{m-1}s'_l -m+\frac{1}{1-\sum_{l=1}^{m-1}\frac{1}{s'_l +m}} .
\end{align*}
Hence $D_{1\cdots m}$ is expressed as 
$$
\left\{ (s'_1,\ldots ,s'_m)\in U \ \left| \ 
\sum_{l=1}^{m-1}s'_l -m+\frac{1}{1-\sum_{l=1}^{m-1}\frac{1}{s'_l +m}}
<s'_m<\frac{m^2}{\frac{1}{m^2}-\vep_0} \cdot \vep_0
\right. \right\} .
$$
Let $\theta$ move from $0$ to $1$, then the arguments of 
$\frac{1}{m^2}-\vep_0 e^{2\pi \sqrt{-1}\theta}$ 
at the start point and the end point are equal. 
Thus the argument of 
$\frac{m^2}{\frac{1}{m^2}-\vep_0 e^{2\pi \sqrt{-1}\theta}} \cdot \vep_0 e^{2\pi \sqrt{-1}\theta}$ 
increases by $2\pi$, when $\theta$ moves from $0$ to $1$. 
Put 
$$
f(s'_1, \ldots ,s'_{m-1}):=
\sum_{l=1}^{m-1}s'_l -m+\frac{1}{1-\sum_{l=1}^{m-1}\frac{1}{s'_l +m}} .
$$
Then $(s'_1 ,\ldots ,s'_{m-1})=(0,\ldots ,0)$ is a critical point of $f$, and 
the Hessian matrix $H_f (0,\ldots ,0)$ at this point is 
positive definite. 
The Morse lemma implies that $f$ is expressed as 
$$
\sum_{l=1}^{m-1} z_l^2 ,
$$
with appropriate coordinates $(z_1 ,\ldots ,z_{m-1})$ 
around the origin. 
Therefore, the claim $\CM_0 (D_{1\cdots m})=
[(-1)^{m-1} \prod_k \ga_k \cdot \al^{-1} \be^{-1} ]\cdot D_{1\cdots m}$ 
is obtained from the following lemma. 
\begin{Lem}
  For $y,\la ,\mu \in \C$, we put 
  \begin{align*}
    & Z_y :=\C^m -\left( \left( z_m-\sum_{l=1}^{m-1} z_l^2 =0\right) \cup (y-z_m=0) \right) 
    \subset \C^m ,\\
    & \nu_y (z):=\left( z_m-\sum_{l=1}^{m-1} z_l^2 \right)^{\la} \cdot (y-z_m)^{\mu} ,
  \end{align*}
  where $z_1,\ldots ,z_m$ are coordinates of $\C^m$. 
  We consider the twisted homology groups $H_m (Z_y ,\nu_y)$ $(y\in \C)$. 
  Let $\de_y \in H_m (Z_y ,\nu_y)\ (y>0)$ be expressed by the twisted cycle 
  defined by the domain 
  $$
  D(y):=
  \left\{ (z_1 ,\ldots ,z_m) \in \R^m \left| \  
  \sum_{l=1}^{m-1} z_l^2 <z_m <y \right. \right\} ,
  $$
  and let $\de'$ be the element in $H_m (Z_1 ,\nu_1)$, which is obtained by 
  the deformation of $\de_1$ along $y=e^{2\pi \sqrt{-1}\theta}$ as $\theta :0 \to 1$. 
  Then we have 
  $$
  \de' =(-1)^{m-1} e^{2\pi \sqrt{-1}(\la +\mu)} \cdot \de_1 .
  $$
\end{Lem}
\begin{proof}
  It is easy to see that the domain $D(y)$ is expressed by 
  $(\xi_1 ,\ldots ,\xi_m)\in [0,1]^m$ as 
  \begin{align*}
    z_l &=(2\xi_l -1)\sqrt{y\xi_m \prod_{j=l+1}^{m-1} (1-(2\xi_j-1)^2)}  
    \quad (1\leq l \leq m-1) ,\\
    z_m &= y\xi_m .
  \end{align*}
  The functions 
  $z_m-\sum_{l=1}^{m-1} z_l^2$ and $y-z_m$ are expressed as
  \begin{align}
    y\xi_m \left( 1-\sum_{l=1}^{m-1}(2\xi_l -1)^2 
      \prod_{j=l+1}^{m-1} \left( 1-(2\xi_j-1)^2 \right) \right), 
    \quad  y(1-\xi_m),
    \label{pull-back}
  \end{align}
  respectively. 
  We consider the variation along $y=e^{2\pi \sqrt{-1}\theta}$ 
  as $\theta :0 \to 1$. 
  The expression of the domain $D(1)$ by 
  $(\xi_1 ,\ldots ,\xi_m)\in [0,1]^m$ is changed. 
  However, by a bijection 
  $$
  r:\ 
  \xi_l \mapsto 1-\xi_l \ (1\leq l \leq m-1),\quad 
  \xi_m \mapsto \xi_m ,
  $$
  the expression coincides with 
  the original one with contributions to orientation. 
  Further, both arguments of $z_m-\sum_{l=1}^{m-1} z_l^2$ and $y-z_m$ 
  increase by $2\pi$, and 
  the expressions (\ref{pull-back}) are invariant under the bijection $r$. 
  Therefore, we obtain 
  $$
  \de' =(-1)^{m-1} e^{2\pi \sqrt{-1}(\la +\mu)} \cdot \de_1 .
  $$
\end{proof}

\appendix
\section{Fundamental group}\label{section-pi1}
In this appendix, we prove Theorem \ref{pi1}. We assume $m \geq 2$.

We regard $\C^m$ as a subset of $\P^m$ and put $L_{\infty} :=\P^m -\C^m$. 
Then we can consider that $S\cup L_{\infty }$ is a hypersurface in $\P^m$, and 
$$
X=\C^m-S =\P^m -(S\cup L_{\infty }).
$$
By a special case of the Zariski theorem of Lefschetz type 
(refer to Proposition (4.3.1) in \cite{Dimca}), 
the inclusion $L-\left( L\cap (S\cup L_{\infty }) \right) \hookrightarrow X$ 
induces a surjection 
$$
\eta : 
\pi_1 \left( L-(L\cap (S\cup L_{\infty })) \right) \to 
\pi_1 (X) ,
$$
for a line $L$ in $\P^m$, which intersects $S\cup L_{\infty }$ transversally 
and avoids its singular parts. 
Note that generators of $\pi_1 (L-(L\cap (S\cup L_{\infty })))$ 
are given by $m+2^{m-1}$ loops going once around
each of the intersection points in $L\cap S \subset \C^m$. 
To define loops in $X$ explicitly, 
we specify such a line $L$ 
in the following way. 
Let $r_1 ,\ldots ,r_{m-1}$ 
be positive real numbers satisfying 
$$
r_1 <\frac{1}{4} ,\quad r_k <\frac{r_{k-1}}{4} \ (2 \leq k \leq m-1),  
$$
and let $\vep=(\vep_1 ,\ldots ,\vep_{m-1})$ be sufficiently small positive real numbers 
such that $\vep_1 <\cdots <\vep_{m-1}$. 
We consider lines 
\begin{align*}
  L_0 : &(x_1 ,\ldots ,x_{m-1},x_m)
  =(r_1 ,\ldots ,r_{m-1} ,0) 
  +t(0 ,\ldots ,0 ,1) \quad (t\in \C), \\
  L_{\vep} : &(x_1 ,\ldots ,x_{m-1},x_m)
  =(r_1 ,\ldots ,r_{m-1} ,0) 
  +t(\vep_1 ,\ldots ,\vep_{m-1} ,1) \quad (t\in \C)
\end{align*}
in $\C^m$. 
We identify $L_{\vep}$ with $\C$ by the coordinate $t$. 
The intersection point 
$L_{\vep} \cap (x_k=0)$ is coordinated by 
$t=-\frac{r_k}{\vep_k} <0$, for $1 \leq k \leq m-1$. 
The intersection point $L_{\vep} \cap (x_m=0)$
is coordinated by $t=0$.
$L_{\vep}$ and $(R(x)=0)$ intersect at $2^{m-1}$ points. 
We coordinate the intersection points $L_{\vep} \cap (R(x)=0)$ by 
$t=t_{a_1 \cdots a_{m-1}}, \ (a_1 ,\ldots ,a_{m-1}) \in \{ 0,1 \} ^{m-1}$. 
The correspondence is as follows. 
We denote the coordinates of the intersection points $L_0 \cap (R(x)=0)$ by 
$$
t_{a_1 \cdots a_{m-1}} ^{(0)}:=
\left( 1+\sum_{k=1}^{m-1} (-1)^{a_k} \sqrt{r_k} \right)^2 .
$$
By this definition, we have 
\begin{align*}
  &t_{a_1 \cdots a_{m-1}}^{(0)} <t_{a'_1 \cdots a'_{m-1}}^{(0)} \\
  &\Longleftrightarrow  
  a_1 -a'_1=\cdots =a_{r-1} -a'_{r-1}=0 ,\ 
  a_r=1,\ a'_r=0 \\
  &\Longleftrightarrow  
  a_1 \cdots a_{m-1} > a'_1 \cdots a'_{m-1} , 
\end{align*}
where $a_1 \cdots a_{m-1}$ is regarded as a binary number. 
For example, if $m=4$, then 
$$
t_{111}^{(0)}<t_{110}^{(0)}<t_{101}^{(0)}<t_{100}^{(0)}<
t_{011}^{(0)}<t_{010}^{(0)}<t_{001}^{(0)}<t_{000}^{(0)}. 
$$
Since $L_{\vep}$ is sufficiently close to $L_0$, 
$t_{a_1 \cdots a_{m-1}}$ is supposed to be arranged near to 
$t_{a_1 \cdots a_{m-1}} ^{(0)}$. 

We can show that $L_0$ does not pass the singular part of $(R(x)=0)$. 
This implies that for sufficiently small $\vep_k$'s, 
$L_{\vep}$ also avoids the singular parts of $S\cup L_{\infty }$. 
Thus, 
$\eta_{\vep} : \pi_1 \left( L_{\vep}-(L_{\vep}\cap (S\cup L_{\infty })) \right) \to \pi_1 (X)$ 
is a surjection. 

Let $\ell_k$ be the 
loop in $L_{\vep}-(L_{\vep} \cap S)$ going once around the intersection point 
$L_{\vep} \cap (x_k=0)$, and let 
$\ell_{a_1 \cdots a_{m-1}}$ be the loop in $L_{\vep}-(L_{\vep} \cap S)$ going 
once around the intersection point $t_{a_1 \cdots a_{m-1}}$. 
Each loop approaches the intersection point through 
the upper half-plane of $t$-space; see Figure \ref{loops}. 
\begin{figure}[h]
  \centering{
    \includegraphics[scale=1.0]{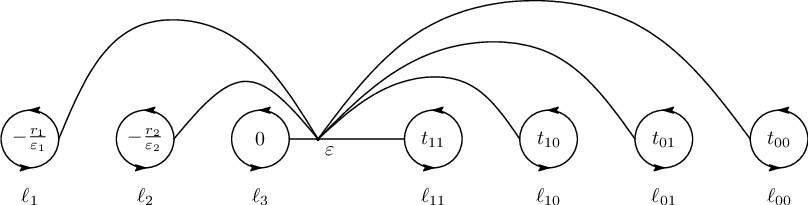} }
  \caption{$\ell_*$ for $m=3$. \label{loops}}
\end{figure}

It is easy to see that 
\begin{align}
  \eta_{\vep} (\ell_k )=\rho_k \ (1\leq k \leq m) ,\quad
  \eta_{\vep} (\ell_{1\cdots 1}) =\rho_0 .
  \label{eta-rho-1}
\end{align}
Further, we have 
$$
\rho_i \rho_j =\rho_j \rho_i \quad (1\leq i,j \leq m) , 
$$
since the fundamental group of $(\C^{\times})^m$ is abelian.
To investigate relations among the $\eta_{\vep}(\ell_{a_1 \cdots a_{m-1}})$'s, 
we consider these loops in $L_0 -(L_0 \cap S)$. 
By the above definition, we can define 
the $\ell_{a_1 \cdots a_{m-1}}$'s as loops in $L_0 -(L_0 \cap S)$. 
Since $L_0$ is sufficiently close to $L_{\vep}$, 
the image of $\ell_{a_1 \cdots a_{m-1}}$ under 
$$
\eta : \pi_1 \left( L_0 -(L_0 \cap (S\cup L_{\infty })) \right) \to \pi_1 (X)
$$
coincides with $\eta_{\vep}(\ell_{a_1 \cdots a_{m-1}})$ 
as elements in $\pi_1 (X)$. 
Though $\eta$ is not a surjection, 
relations among the $\eta (\ell_{a_1 \cdots a_{m-1}})$'s in $\pi_1 (X)$ 
can be regarded as those among the $\eta_{\vep}(\ell_{a_1 \cdots a_{m-1}})$'s. 

\begin{Lem}\label{lemma}
  \begin{enumerate}[{\rm (i)}]
  \item $\eta(\ell_{a_1 \cdots a_{k-1} 0 a_{k+1} \cdots a_{m-1}})
    =\rho_k \eta(\ell_{a_1 \cdots a_{k-1} 1 a_{k+1} \cdots a_{m-1}}) \rho_k^{-1}$. 
  \item $\eta(\ell_{1 \cdots 1}) 
    =\rho_{m-1} \eta(\ell_{1\cdots 1} \ell_{1\cdots 10} \ell_{1\cdots 1}^{-1}) 
    \rho_{m-1}^{-1}$.
  \end{enumerate}
\end{Lem}
Temporarily, we admit this lemma. 
By (i), we have 
\begin{align}
  \label{eta-rho-2}  
  \eta_{\vep}(\ell_{a_1 \cdots a_{m-1}})
  =\eta(\ell_{a_1 \cdots a_{m-1}}) 
  &=(\rho_1^{b_1} \cdots \rho_{m-1}^{b_{m-1}}) \cdot 
  \eta(\ell_{1\cdots 1}) \cdot
  (\rho_1^{b_1} \cdots \rho_{m-1}^{b_{m-1}})^{-1} 
  \\
  \nonumber 
  &=(\rho_1^{b_1} \cdots \rho_{m-1}^{b_{m-1}}) \cdot 
  \rho_0 \cdot
  (\rho_1^{b_1} \cdots \rho_{m-1}^{b_{m-1}})^{-1} 
\end{align}
as elements in $\pi_1(X)$, 
where $(b_1,\ldots ,b_{m-1}):=(1-a_1,\ldots ,1-a_{m-1})$. 
This implies that the loops $\rho_0,\ldots ,\rho_m$ generate $\pi_1 (X)$, 
since the images of the $\ell_k$'s and $\ell_{a_1 \cdots a_{m-1}}$'s by $\eta$ 
generate $\pi_1 (X)$. 
By (ii) and the above argument, we obtain 
\begin{align*}
  \rho_0 =\eta(\ell_{1\cdots 1})
  &=\rho_{m-1} \eta(\ell_{1\cdots 1} \ell_{1\cdots 10} \ell_{1\cdots 1}^{-1}) 
  \rho_{m-1}^{-1} \\
  &=\rho_{m-1} \cdot \rho_0 \cdot \rho_{m-1} \rho_0 \rho_{m-1}^{-1} \cdot
  \rho_0^{-1} \cdot \rho_{m-1}^{-1}, 
\end{align*}
that is, $(\rho_0 \rho_{m-1})^2 =(\rho_{m-1} \rho_0)^2$. 
Changing the definitions of $L_0$ and $L_{\vep}$, 
we obtain the relations 
$$
(\rho_0 \rho_k)^2 =(\rho_k \rho_0)^2 \quad (1\leq k \leq m).
$$
For example, if we put 
$$
L_{\vep} : (x_1 ,x_2, \ldots ,x_m)
=(0,r_1 ,\ldots ,r_{m-1} ) 
+t(1, \vep_1 ,\ldots ,\vep_{m-1} ) \quad (t\in \C) ,
$$
then a similar argument shows 
$(\rho_0 \rho_m)^2 =(\rho_m \rho_0)^2$. 
Therefore, the proof of Theorem \ref{pi1} is completed. 

\begin{proof}[Proof of Lemma \ref{lemma}]
  For $\theta \in [0,1]$, 
  let $L(\theta)$ be the line defined by 
  \begin{align*}
    L(\theta) :&(x_1 ,\ldots ,x_k ,\ldots ,x_{m-1} ,x_m) \\
    &=(r_1 ,\ldots ,e^{2\pi \sqrt{-1}\theta }r_k ,\ldots ,r_{m-1} ,0) 
    +t(0 ,\ldots ,0 ,1) \quad (t\in \C) .
  \end{align*}
  Note that $L(0)=L(1)=L_0$. 
  We identify $L(\theta)$ with $\C$ by the coordinate $t$. 
  It is easy to see that the intersection points of 
  $L(\theta)$ and $(R(x)=0)$ are given by the following $2^{m-1}$ elements:
  $$
  t_{a_1 \cdots a_{m-1}} ^{(\theta )}:=
  \left( 1+\sum_{\substack{j=1\\j \neq k}}^{m-1} (-1)^{a_j} \sqrt{r_j} 
  +(-1)^{a_k} \sqrt{r_k}e^{\pi \sqrt{-1}\theta} \right)^2 .
  $$
  The points $1+\sum_{j \neq k} (-1)^{a_j} \sqrt{r_j} 
  +(-1)^{a_k} \sqrt{r_k}e^{\pi \sqrt{-1}\theta}$ are in the right half-plane 
  for any $\theta \in [0,1]$, 
  since $\sum_{j=1}^{m-1} \sqrt{r_j} <\sum_{j=1}^{m-1} 2^{-j}<1$. 
  Let $\theta$ move from $0$ to $1$, then 
  \begin{enumerate}[(a)]
  \item $t_{a_1 \cdots a_{k-1} 0 a_{k+1} \cdots a_{m-1}}^{(1)}
    =t_{a_1 \cdots a_{k-1} 1 a_{k+1} \cdots a_{m-1}}^{(0)}$,  
    $t_{a_1 \cdots a_{k-1} 1 a_{k+1} \cdots a_{m-1}}^{(1)}
    =t_{a_1 \cdots a_{k-1} 0 a_{k+1} \cdots a_{m-1}}^{(0)}$, 
  \item $t_{a_1 \cdots a_{k-1} 0 a_{k+1} \cdots a_{m-1}}^{(\theta)}$ 
    moves in the upper half-plane, 
  \item $t_{a_1 \cdots a_{k-1} 1 a_{k+1} \cdots a_{m-1}}^{(\theta)}$ 
    moves in the lower half-plane. 
  \end{enumerate}
  For example, the $t_{a_1 a_2 a_3}$'s move as Figure \ref{point-change}, 
  for $m=4$ and $k=2$. 
  \begin{figure}[h]
    \centering{
      \includegraphics[scale=1.0]{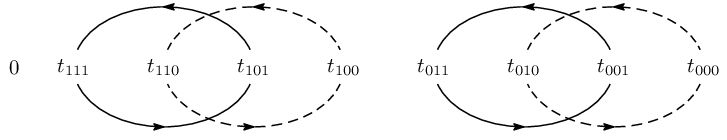} }
    \caption{$t_{a_1 a_2 a_3}$ for $m=4,\ k=2$. \label{point-change}}
  \end{figure}

  We put $P(\theta) :=\C -\{ t_{a_1 \cdots a_{m-1}}^{(\theta)} \mid a_j \in \{ 0,1 \} \}$ 
  that is regarded as a subset of $L(\theta )$. 
  Let $\vep'$ be a sufficiently small positive real number, 
  and we consider the fundamental group $\pi_1 (P(\theta),\vep')$. 
  As mentioned above, the $\ell_{a_1 \cdots a_{m-1}}$'s are defined as 
  elements in $\pi_1 (P(0),\vep')=\pi_1 (P(1),\vep')$. 
  Let $\theta$ move from $0$ to $1$, 
  then the $\ell_{a_1 \cdots a_{m-1}}$'s 
  define the elements in each $\pi_1 (P(\theta),\vep')$ naturally. 
  The properties (a), (b), (c) imply the following. 
  \begin{Lem}
    \label{lem-loop-change}
    $\ell_{a_1 \cdots a_{k-1} 0 a_{k+1} \cdots a_{m-1}}$ in $\pi_1 (P(0),\vep')$ 
    changes into $\ell_{a_1 \cdots a_{k-1} 1 a_{k+1} \cdots a_{m-1}}$ 
    in $\pi_1 (P(1),\vep')$. 
  \end{Lem}
  We give the proof of this lemma below. 
  By this variation, the base point moves around the divisor $(x_k=0)$, 
  since the base point $\vep' \in P(\theta)$ corresponds to the point 
  $(r_1 ,\ldots ,e^{2\pi \sqrt{-1}\theta }r_k ,\ldots ,r_{m-1} ,\vep') \in L(\theta)$. 
  It implies the conjugation by $\rho_k$ in $\pi_1 (X)$. 
  Hence we obtain the relation (i).

  To prove (ii), we use a similar argument for 
  $k=m-1$ and $\ell_{1\cdots 1} \in \pi_1 (P(0),\vep')$. 
  Let $\theta$ move from $0$ to $1$, then $\ell_{1\cdots 1}$ changes into 
  a loop in $P(1)$, which goes once around $t_{1\cdots 1}^{(1)}=t_{1\cdots 10}^{(0)}$ and 
  approaches this point through the lower half-plane (see Figure \ref{loop-change}). 
  Since such a loop is homotopic to  
  $\ell_{1\cdots 1} \ell_{1\cdots 10} \ell_{1\cdots 1}^{-1}$, 
  we obtain (ii). 
  \begin{figure}[h]
    \centering{
      \includegraphics[scale=1.0]{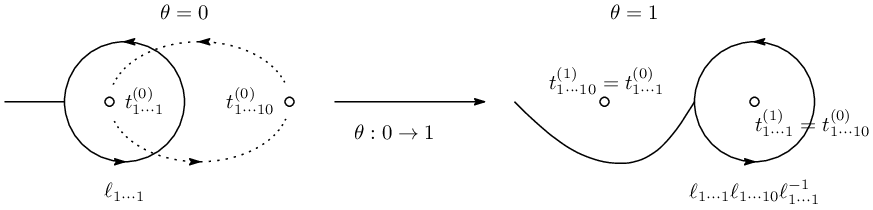} }
    \caption{the variation of $\ell_{1\cdots 1}$. \label{loop-change}}
  \end{figure}
\end{proof}
\begin{proof}[Proof of Lemma \ref{lem-loop-change}]
  We show that the variations of 
  $t_{a'_1 \cdots a'_{m-1}}$'s  
  do not interfere with moving of the loop $\ell_{a_1 \cdots a_{k-1} 0 a_{k+1} \cdots a_{m-1}}$.  
  We put $\tilde{t}_{a_1 \cdots a_{m-1}}^{(\theta)}:=
  1+\sum_{j \neq k} (-1)^{a_j} \sqrt{r_j}+(-1)^{a_k} \sqrt{r_k}e^{\pi \sqrt{-1}\theta}$. 
  This satisfies $(\tilde{t}_{a_1 \cdots a_{m-1}}^{(\theta)})^2=t_{a_1 \cdots a_{m-1}}^{(\theta)}$. 
  Since each $\tilde{t}_{a_1 \cdots a_{m-1}}^{(\theta)}$ is in the right half-plane, 
  $t_{a_1 \cdots a_{m-1}}^{(\theta)}$ does not meet the half-line $(-\infty, 0] \subset \R$. 
  For each $\theta$, 
  $\tilde{P}(\theta) :=$ (the right half-plane) 
  $-\{ \tilde{t}_{a_1 \cdots a_{m-1}}^{(\theta)} \mid a_j \in \{ 0,1 \} \}$ 
  is homeomorphic to $P(\theta)-( -\infty ,0 ]$ by the map 
  $$
  h: \tilde{P}(\theta) \longrightarrow P(\theta)-( -\infty ,0 ];\quad 
  z \longmapsto z^2. 
  $$
  It is sufficient to show that the points $\tilde{t}_{a_1 \cdots a_{m-1}}^{(\theta)}$'s 
  do not interfere with moving of 
  the loop 
  $\tilde{\ell}_{a_1 \cdots a_{k-1} 0 a_{k+1} \cdots a_{m-1}}$ in $\tilde{P}(\theta)$, 
  which satisfies $h_* (\tilde{\ell}_{a_1 \cdots a_{k-1} 0 a_{k+1} \cdots a_{m-1}})
  =\ell_{a_1 \cdots a_{k-1} 0 a_{k+1} \cdots a_{m-1}}$. 
  Since each $\tilde{t}_{a'_1 \cdots a'_{k-1} 1 a'_{k+1} \cdots a'_{m-1}}^{(\theta)}$ moves in 
  lower half-plane, it does not interfere with moving of 
  $\tilde{\ell}_{a_1 \cdots a_{k-1} 0 a_{k+1} \cdots a_{m-1}}$. 
  We consider the variation of $\tilde{t}_{a'_1 \cdots a'_{k-1} 0 a'_{k+1} \cdots a'_{m-1}}^{(\theta)}$ 
  for $(a'_1 ,\ldots ,a'_{k-1} , a'_{k+1} ,\ldots ,a'_{m-1}) \neq (a_1 ,\ldots ,a_{k-1} , a_{k+1} ,\ldots ,a_{m-1})$. 
  By the definition, 
  $\tilde{t}_{a'_1 \cdots a'_{k-1} 0 a'_{k+1} \cdots a'_{m-1}}^{(\theta)}
  -\tilde{t}_{a_1 \cdots a_{k-1} 0 a_{k+1} \cdots a_{m-1}}^{(\theta)}$ 
  does not depend on $\theta$. 
  Thus, $\tilde{t}_{a'_1 \cdots a'_{k-1} 0 a'_{k+1} \cdots a'_{m-1}}^{(\theta)}$ 
  moves parallel to $\tilde{t}_{a_1 \cdots a_{k-1} 0 a_{k+1} \cdots a_{m-1}}^{(\theta)}$. 
  This implies $\tilde{t}_{a'_1 \cdots a'_{k-1} 0 a'_{k+1} \cdots a'_{m-1}}^{(\theta)}$ 
  does not interfere with moving of 
  $\tilde{\ell}_{a_1 \cdots a_{k-1} 0 a_{k+1} \cdots a_{m-1}}$. 
  Therefore, the proof is completed. 
\end{proof}

\end{document}